\documentclass{amsart}
\usepackage[T1]{fontenc}
\usepackage{geometry}
\usepackage{color}
\usepackage{float}
\usepackage{amsmath}
\usepackage{amssymb}
\usepackage{amsthm}
\usepackage{mathrsfs}
\usepackage{setspace}
\usepackage[all,cmtip]{xy}
\usepackage{tikz}
\usepackage{enumerate}
\usepackage{mathtools}

\newtagform{roman}[]()

\newcommand{\Q}{{\mathbb{Q}}}

\newcommand{\fm}{{\mathfrak m}}
\newcommand{\fn}{{\mathfrak n}}
\newcommand{\fp}{{\mathfrak p}}
\newcommand{\fq}{{\mathfrak q}}
\newcommand{\fd}{{\mathfrak d}}
\newcommand{\fe}{{\mathfrak e}}
\newcommand{\frr}{{\mathfrak r}}
\newcommand{\ft}{{\mathfrak t}}

\newcommand{\cA}{{\mathcal A}}
\newcommand{\cD}{{\mathcal D}}
\newcommand{\cO}{{\mathscr O}}

\newcommand{\ep}{{\varepsilon}}

\newcommand{\rN}{{\mathrm {N}}}

\newtheorem{thm}{Theorem}[section]

\newtheorem{cor}[thm]{Corollary}
\newtheorem{lem}[thm]{Lemma}

\theoremstyle{definition}
\newtheorem{defn}[thm]{Definition}
\newtheorem{theorem}{Theorem}

\numberwithin{equation}{section}

\begin{document}

\title[A Generalization on the Barban-Vehov Problem]{A Generalization of Graham's Estimate \\ on the Barban-Vehov Problem}

\author[Chen An]{Chen An}
\address{Duke University, 120 Science Drive, Durham NC 27708}
\email{chen.an@duke.edu, chen.an.nku@gmail.com}

\maketitle

\begin{abstract}
Suppose $\{ \lambda_d\}$ are Selberg's sieve weights and $1 \le w < y \le x$. Graham's estimate on the Barban-Vehov problem shows that $\sum_{1 \le n \le x} (\sum_{d|n} \lambda_d)^2 = \frac{x}{\log(y/w)} + O(\frac{x}{\log^2(y/w)})$.
We prove an analogue of this estimate for a sum over ideals of an arbitrary number field $k$.
Our asymptotic estimate remains the same; the only difference is that the effective error term may depend on arithmetics of $k$. Our innovation involves multiple counting results on ideals instead of integers. Notably, some of the results are nontrivial generalizations. Furthermore, we prove a corollary that leads to a new zero density estimate.
\end{abstract}




\section{Introduction}

Fix $1 \le w < y$. For any positive integer $d$, define
\[ m(d):=\begin{cases} 1 & d \le w, \\ \frac{\log(y/d)}{\log(y/w)} & w < d \le y, \\ 0 & d>y, \end{cases} \]
and define $\lambda_d := \mu(d)m(d)$, for the M\"{o}bius function $\mu(d)$. The set $\{\lambda_d \}_{d \ge 1}$ is the set of Selberg sieve weights.
Barban and Vehov \cite{BV68} proposed a problem on estimating the sum $\sum_{1 \le n \le x} (\sum_{d|n} \lambda_d)^2$ for any $x \ge y$. They proved an upper bound
$ \ll \frac{x}{\log (y/w)}.$
This bound was made asymptotic by Graham \cite{Gr78}, who showed
\begin{equation}\label{graham}
 \sum_{1 \le n \le x} \left(\sum_{d|n} \lambda_d \right)^2 = \frac{x}{\log (y/w)} + O\left(\frac{x}{\log^2 (y/w)} \right). 
 \end{equation}
 Equation (\ref{graham}) has applications to zero density estimates and primes in arithmetic progressions, including Linnik's constant. Further generalizations on the Barban-Vehov problem have been made by, e.g., Jutila \cite{Ju79}, and Murty \cite{Mur19}. Such improvements either regard generalizing the summand or imposing a congruence condition on the sum over $n$. 

The purpose of this paper is to generalize the asymptotic estimate (\ref{graham}) to an arbitrary base field $k$ rather than the base field $\Q$, i.e., the sum is taken over ideals in $\cO_k$ instead of positive integers. In particular, we will define the Selberg sieve weights for ideals and prove the following Theorem \ref{igtoug1}. Our aim to generalize the Barban-Vehov problem is three-fold. First, Theorem \ref{igtoug1} leads to a new zero density estimate via Corollary \ref{ug}. Second, it is a natural but nontrivial generalization from a sum of integers to a sum of ideals. During the proof of Theorem \ref{igtoug1}, we need to introduce multiple lemmas on estimating arithmetic functions related to ideals. Third, our Theorem \ref{igtoug1} and Corollary \ref{ug} below are expected to have wide applications to methods involving the large sieve, over an arbitrary number field.


For any ideal $\fn$ of $\cO_k$, define 
\[ \mu_k(\fn):=\begin{cases} (-1)^s & \text{ if } \fn \text{ is a product of } s \text{ distinct prime ideals,} \\ 0 & \text{ otherwise.} \end{cases} \] Define 
\[ \lambda_\fn := \mu_k(\fn)m(\rN\fn). \]
Denote $\fd | \fn$ if $\fn \subseteq \fd$ and denote $n_k=[k:\Q],d_k=\text{disc}(k/\Q)$. Denote $\beta_0$ to be the possible (real, simple) Siegel zero for $k$ (see Definition \ref{Siegel} for its definition). In particular, $\frac12 < \beta_0 < 1$.

Our main theorem is the following.



\begin{thm}\label{igtoug1}
Let $k$ be a number field of degree $n_k$ and discriminant $d_k$, and let $1 \le w < y \le x$. 
Then we have
\[ \sum_{\substack{\fn \subseteq \cO_k, \rN\fn \le x}} \left( \sum_{\fd | \fn} \lambda_\fd \right)^2 = \frac{x}{\log(y/w)}+O_{n_k,d_k}(\frac{C_{\beta_0}x}{\log^2(y/w)}). \]
where we may take $C_{\beta_0}=\frac{1}{(1-\beta_0)^6D_{\beta_0}^2}$ with $D_{\beta_0}=\min \{1,|\zeta_k^{'}(\beta_0)| \}$ if $\beta_0$ exists and $C_{\beta_0}=1$ otherwise.
\end{thm}



Theorem \ref{igtoug1} appears to be the first theorem that provides an asymptotic estimate to the Barban-Vehov problem for a base field other than $\Q$. When the base field is $\Q$, Graham's estimate (\ref{graham}) is currently the best bound,  and the quality of our asymptotic estimate matches the one given by Graham. For the proof of Theorem \ref{igtoug1}, we follow the outline of Graham's approach while proving a series of new estimates with modified methods. These methods allow us to overcome obstacles when we work over an arbitrary base field. A few highlights of these methods are in the proofs of Lemmas \ref{muandphi}, \ref{sqfrc}, and \ref{678ing}. Our notation differs from Graham since our notation naturally derives from an application of zero density estimates in our work on the large sieve in \cite{An20}. However, there is a correspondence: our $x,w,y$ correspond to Graham's $N,z_1,z_2$.

Theorem \ref{igtoug1} leads to the following Corollary \ref{ug}, which is a crucial ingredient to prove a new zero density estimate. The approach from Corollary \ref{ug} to the zero density estimate is proved in the companion paper \cite{An20}.

\begin{cor}\label{ug}
Let $k$ be a number field of degree $n_k$ and discriminant $d_k$, and let $1 \le w < y \le x$. 
For any $\alpha$ with $1/2<\alpha<1$, 
\begin{equation*}
\sum_{\substack{\fn \subseteq \cO_k, \rN\fn \le x}}  \left( \sum_{\fd | \fn} \lambda_\fd \right)^2\rN\fn^{1-2\alpha} \ll_{n_k,d_k} C_{\beta_0} \frac{\log(x/w)}{\log(y/w)}x^{2-2\alpha},
\end{equation*}
where $C_{\beta_0}$ is as in Theorem \ref{igtoug1}.
\end{cor}

When the base field is $\Q$, this recovers Lemma 9 in \cite{Gr81}. 
We reserve the proof of Corollary \ref{ug} to the end of the paper, in Section \ref{pflem}. 



Now, we outline the proof of Theorem \ref{igtoug1}. It is a consequence of the following Theorem \ref{ig}; we deduce Theorem \ref{igtoug1} from Theorem \ref{ig} in Section \ref{pflem}.

Let $w<y$ be as in Theorem \ref{igtoug1}. For ideals $\fd,\fe$ of $\cO_k$, define
\begin{equation}\label{lambda12}
 \Lambda_1(\fd) = \begin{cases} \mu_k(\fd) \log(\frac{w}{\rN\fd}) & \text{ if } \rN\fd \le w, \\ 0 & \text{ if } \rN\fd > w, \end{cases} \ \ \ \ \ \  \ \ \ \ \ \ \ \Lambda_2(\fe) = \begin{cases} \mu_k(\fe) \log(\frac{y}{\rN\fe}) & \text{ if } \rN\fe \le y, \\ 0 & \text{ if } \rN\fe > y. \end{cases} 
 \end{equation}

\begin{thm}\label{ig}
Let $k$ be a number field of degree $n_k$ and discriminant $d_k$, and let $1 \le w < y \le x$. We have
\[ \sum_{\rN\fn \le x} \left( \sum_{\fd | \fn} \Lambda_1(\fd) \right) \left( \sum_{\fe | \fn} \Lambda_2(\fe) \right) = x\log w + O_{n_k,d_k}\left(C_{\beta_0}x\right), \]
where $C_{\beta_0}$ is as in Theorem \ref{igtoug1}.
\end{thm}

If $n_k=1$, Theorem \ref{ig} recovers Theorem in \cite[p.84]{Gr78}. So we suppose $n_k \ge 2$ throughout the paper. 
To compute the left-hand side in Theorem \ref{ig}, we follow Graham \cite{Gr78} to consider the cases $wy \le x$ and $wy>x$ separately. In the former case, it is straightforward to obtain our estimate. In the latter case, a small arithmetic trick is needed before our computation (see (\ref{smalltrick})), generalizing a trick developed in Graham's work. The rest of the computation relies heavily on the series of new estimates on arithmetics related to ideals. Among these estimates, the generalized M\"{o}bius function $\mu_k(\fn)$ plays a vital role. Its oscillating nature makes our method possible; see the statements from Lemma \ref{678ing} to Lemma \ref{12ing}.
Our argument's error terms may have dependencies on $n_k,d_k$, which originate in our estimates. In our work below, we omit the notation on these dependencies for simplicity.

The outline of our paper is as follows. In Section \ref{nots}, we provide a list of notations. In Section \ref{lems}, we provide all necessary lemmas in the paper. In Section \ref{wylex}, we prove Theorem \ref{ig} in the case $wy \le x$. In Section \ref{wygx}, we prove Theorem \ref{ig} in the case $wy > x$. In Section \ref{pflem}, we deduce Theorem \ref{igtoug1} from Theorem \ref{ig}, deduce Corollary \ref{ug} from Theorem \ref{igtoug1}, and prove all the lemmas.

\section{Notations}\label{nots}

In this section, we provide a list of notations and functions supported on ideals in $\cO_k$. These functions are generalizations of standard arithmetic functions. All the Gothic letters are ideals and all the ideals $\fp$ are prime ideals.

\begin{itemize}
\item
squarefree ideals = ideals $\fn$ of the form $\fn = \fp_1\fp_2\dots\fp_s$, where $\fp_1,\dots,\fp_s$ are distinct prime ideals
\item
$\omega(\fn)=$ number of distinct prime ideal divisors of $\fn$
\item
$(\fd,\fe)=$ the smallest ideal that contains both $\fd$ and $\fe$; $\fd,\fe$ are relatively prime means $(\fd,\fe)=(1)$
\item
$[\fd,\fe]=$ the largest ideal that is contained in both $\fd$ and $\fe$
\item
Euler totient function for ideals $\varphi(\fn)=\rN\fn \prod_{\fp | \fn} (1-\frac{1}{\rN\fp})$
\item
divisor function for ideals $d(\fn)=\sum_{\fd | \fn} 1=$ the number of divisors of $\fn$
\item
Liouville function for ideals $\lambda(\fn)=(-1)^{k_1+\dots+k_s}$ if $\fn=\fp_1^{k_1} \dots \fp_s^{k_s}$, where $\fp_1,\dots,\fp_s$ are distinct prime ideals
\item
von Mangoldt function for ideals
\[  \Lambda(\fn) =  \begin{cases} \log (\rN\fp), & \text{ if } \fn = \fp^m \text{ for some positive integer }m \\ 0, & \text{ otherwise} \end{cases}  \]
\item
$\kappa(\fn) = \rN\fn \prod_{\fp | \fn} (1+\frac{1}{\rN\fp})$
\item
$\sigma_{a}(\fn) = \sum_{\fd | \fn} \rN\fd^{a}$
\end{itemize}

\section{Statements of All Lemmas}\label{lems}

In this section, we state all the technical lemmas used later in our proof. All of the lemmas are counting results on arithmetic objects related to ideals. The lemmas up to Lemma \ref{0108} are non-oscillating results, i.e., summing over nonnegative values, whereas the lemmas from Lemma \ref{678ing} onward are oscillating results, since the sum involves $\mu_k$. 

Let $Q \ge 1$ be a real number. In all the results below, the number $Q$ always serves as the bound of ideal norms. We first provide a theorem that is a weaker version of \cite[Theorem 3]{LTT22}, and all the other lemmas elaborate on statements of this type.

\begin{theorem}\label{counting}\cite[Theorem 3 (weak version)]{LTT22}
There exists an absolute constant $s_k>0$ such that
\[ \sum_{\rN\fn \le Q} 1 = s_kQ+O(Q^{1-\frac{1}{n_k}}), \]
where
\[ s_k=\frac{2^{r_1}(2\pi)^{r_2}hR}{w\sqrt{|d_k|}}, \] $r_1$ is the number of real embeddings of $k$ within a fixed choice of $\overline{\Q}$, $r_2$ is the number of pairs of complex embeddings of $k$ within $\overline{\Q}$, $h$ is the class number of $k$, $R$ is the regulator, $w$ is the number of roots of unity, and the constant associated to the error term is effective. 
\end{theorem}
Using the class number formula, we have $s_k=\mathrm{Res}_{s=1} \zeta_k(s)$.
Our error term in Theorem \ref{counting} is weaker than \cite[Theorem 3]{LTT22} but is sufficient for our purpose.

Now we are ready to state our lemmas. The proofs are reserved to Section \ref{pflem}.

\begin{lem}\label{power}
For any constant $0<a \le 1$, we have
\begin{equation*}
\sum_{\substack{\rN\fn \le Q}} \rN\fn^{-a}=\begin{cases}
 \frac{s_k}{1-a}Q^{1-a}+ O(Q^{1-\frac{1}{n_k}-a})+O(1) & \text{ if } 0<a< 1-\frac{1}{n_k} \text{ or } 1-\frac{1}{n_k} < a <1, \\
 s_kn_kQ^{\frac{1}{n_k}} + O(\log Q) & \text{ if } a = 1-\frac{1}{n_k}, \\
 s_k \log Q + O(1) & \text{ if } a=1.
\end{cases}.
\end{equation*}
\end{lem}

\begin{lem}\label{logpower}
For the constant $a=1-\frac{1}{n_k}$, we have
\begin{equation*}
\sum_{\substack{\rN\fn \le Q}} \frac{\log(\frac{Q}{\rN\fn})}{\rN\fn^{a}} = s_kn_k^2Q^{\frac{1}{n_k}}+O((\log Q)^2).
\end{equation*}
\end{lem}

Recall the function $\sigma_a$ defined in Section \ref{nots}.
\begin{lem}\label{01073}
If $a<0$, then
\[ \sum_{\rN\fn \le Q} \sigma_a(\fn) \le \sum_{\rN\fn \le Q} \sigma_a^2(\fn) \ll_a Q. \]
\end{lem}

\begin{lem}\label{loginv}
We have
\[ \sum_{\rN\fn \le Q} \frac{1}{\rN\fn (\log(\frac{2Q}{\rN\fn}))^2} = O(1). \]
For any $a<0$ and $0<b \le 1$,
\[ \sum_{\rN\fn \le Q} \frac{\sigma_{a}^2(\fn)}{\rN\fn^b (\log(\frac{2Q}{\rN\fn}))^2} = O(Q^{1-b}). \]
\end{lem}

In our proof of Theorem \ref{ig}, squarefree ideals are important. The next three lemmas involve such ideals.

\begin{lem}\label{muandphi}
We have
\[ \sum_{\rN\fn \le Q} \frac{\mu_k^2(\fn)}{\varphi(\fn)} = s_k\log Q + O(1). \]
\end{lem}


\begin{lem}\label{sqfrc}
We have
\[ \sum_{\rN\fn \le Q} \mu_k^2(\fn) = \frac{s_k}{\zeta_k(2)}Q+O(Q^{1-\frac{1}{n_k}+\ep_k}), \]
where $\ep_k = 0$ for $n_k \ge 3$ and $\ep_k>0$ is arbitrarily small for $n_k = 2$.
\end{lem}

\begin{lem}\label{0108}
For any squarefree integral ideal $\frr$ in $\cO_k$, we have
\[ \sum_{\substack{\rN\fn \le Q \\ (\fn,\frr) = (1)}} \mu_k^2(\fn) = \frac{s_k}{\zeta_k(2)} \frac{\rN\frr}{\kappa(\frr)}Q +  O(Q^{1-\frac{1}{n_k}+\ep_k} \sigma_{-\frac{1}{n_k+1}}(\frr)), \]
where $\ep_k$ is as in Lemma \ref{sqfrc}.
\end{lem}

The six statements in the next four lemmas have their respective analogues as (6), (7), (8), (10), (11), and (12) in Graham's paper \cite{Gr78}. In all the lemmas above, the implied constants are effectively computable.
Before we state the lemmas, we need to introduce a zero-free region for the Dedekind zeta function $\zeta_k(s)$. 

\begin{theorem}\label{ik04}\cite[Theorem 5.33]{IK04}
There exists an absolute constant $c>0$ such that $\zeta_k(s) (s=\sigma+it)$ has no zero in the region
\[  \sigma \ge 1-\frac{c}{n_k^2 \log |d_k| (|t|+3)^{n_k}}, \]
except possibly a simple real zero $\beta_0 < 1$.
\end{theorem}

\begin{defn}\label{Siegel}
If the zero $\beta_0$ exists, then it is called a Siegel zero.
\end{defn}

Our next lemmas now explicitly denote any (possible) dependence on $\beta_0$. 


\begin{lem}\label{678ing}
For any $A>0$, the following bounds hold:
\begin{equation}\label{6ing}
\sum_{\rN\fn \le Q} \mu_k(\fn)- \frac{Q^{\beta_0}}{\beta_0\zeta_k^{'}(\beta_0)} \ll_A Q(\log 2Q)^{-A}.
\end{equation}
The term involving ${\beta_0}$ exists if and only if the Siegel zero exists (and similarly for the expressions below). Second,
\begin{equation}\label{7ing}
\sum_{\rN\fn \le Q} \frac{\mu_k(\fn)}{\rN\fn}-\frac{Q^{\beta_0-1}}{(\beta_0-1)\zeta_k^{'}(\beta_0)} \ll_A (\log 2Q)^{-A}.
\end{equation}
Third,
\begin{equation}\label{8ing}
 \sum_{\rN\fn \le Q} \frac{\mu_k(\fn) \log \rN\fn}{\rN\fn} = -\frac{1}{s_k} +\frac{Q^{\beta_0-1}[(\beta_0-1)\log Q-1]}{(\beta_0-1)^2\zeta_k^{'}(\beta_0)}+O_A((\log 2Q)^{-A}). 
\end{equation}

\end{lem}

For convenience, define 
\[ C_{\beta_0}^{'}:=\begin{cases} \max\{1,\frac{1}{(\beta_0-1)^2|\zeta_k^{'}(\beta_0)|}\} & \text{ if } \beta_0 \text{ exists}, \\ 1 & \text{otherwise}. \end{cases} \] This quantity will appear frequently in our bounds.

\begin{lem}\label{10ing}
For any squarefree integral ideal $\frr$ in $\cO_k$ and any $A>0$,
\[  \sum_{\substack{\rN\fn \le Q \\ (\fn,\frr) = (1)}} \frac{\mu_k(\fn)}{\rN\fn} \log(\frac{Q}{\rN\fn}) = \frac{\rN\frr}{s_k\varphi(\frr)} +C_{\frr,\beta_0}Q^{\beta_0-1}+ O_A(C_{\beta_0}^{'}\sigma_{-\frac12}(\frr)(\log 2Q)^{-A}), \]
where $C_{\frr,\beta_0}$ is a constant depending only on $\frr,\beta_0$, and $|C_{\frr,\beta_0}| \le C_{\beta_0}^{'}\sigma_{-\frac12}(\frr)$.
\end{lem}

Recall the definition of $\kappa(\fn)$ in Section \ref{nots}.

\begin{lem}\label{11ing}
For any squarefree integral ideal $\frr$ in $\cO_k$ and any $A>0$,
\[ \sum_{\substack{\rN\fn \le Q \\ (\fn,\frr) = (1)}} \frac{\mu_k(\fn)}{\kappa(\fn)} \log(\frac{Q}{\rN\fn}) = \frac{\kappa(\frr)}{\rN\frr} \frac{\zeta_k(2)}{s_k} +C_{\frr,\beta_0}^{'}Q^{\beta_0-1}+ O_A(C_{\beta_0}^{'}\sigma_{-\frac12}(\frr) (\log 2Q)^{-A}), \]
where $C_{\frr,\beta_0}^{'}$ is a constant depending only on $\frr,\beta_0$, and $|C_{\frr,\beta_0}^{'}| \ll C_{\beta_0}^{'}\sigma_{-\frac12}(\frr)$.
\end{lem}

\begin{lem}\label{12ing}
For any squarefree integral ideal $\frr$ in $\cO_k$ and any $A>0$,
\[ \sum_{\substack{\rN\fn \le Q \\ (\fn,\frr) = (1)}} \frac{\mu_k(\fn)}{\kappa(\fn)} = O_{A}(C_{\frr,\beta_0}^{'}(1-\beta_0)Q^{\beta_0-1} +C_{\beta_0}^{'}\sigma_{-\frac12}(\frr) (\log 2Q)^{-A}), \]
where $C_{\frr,\beta_0}^{'}$ is as in Lemma \ref{11ing}.
\end{lem}

All these lemmas are proved in Section \ref{pflem}. We now assume the above lemmas and prove Theorem \ref{ig}.

\section{Proof of Theorem \ref{ig} in the case $wy \le x$}\label{wylex}

Recall the functions $\Lambda_1$ and $\Lambda_2$ defined in (\ref{lambda12}).
We write
\begin{eqnarray}\label{wysmall}
\nonumber \sum_{1 < \rN\fn \le x} \left( \sum_{\fd | \fn} \Lambda_1(\fd) \right) \left( \sum_{\fe | \fn} \Lambda_2(\fe) \right) &=& \sum_{1 < \rN\fn \le x} \sum_{[\fd,\fe]|\fn} \Lambda_1(\fd) \Lambda_2(\fe) \\
&=& \sum_{\substack{\fd,\fe \\ \rN([\fd,\fe]) \le x}} \Lambda_1(\fd) \Lambda_2(\fe) \sum_{\substack{\fn \\ [\fd,\fe]|\fn \\ 1<\rN\fn \le x}} 1.
\end{eqnarray}
So we can apply Theorem \ref{counting} and write the sum in (\ref{wysmall}) as
\begin{equation}\label{1me}
\sum_{\substack{\fd,\fe \\ \rN\fd \le w \\ \rN\fe \le y}} \Lambda_1(\fd) \Lambda_2(\fe) \frac{s_kx}{\rN([\fd,\fe])} + O(x^{1-\frac{1}{n_k}} \sum_{\substack{\fd,\fe \\ \rN\fd \le w \\ \rN\fe \le y}} \frac{\log(\frac{w}{\rN\fd}) \log(\frac{y}{\rN\fe})}{\rN([\fd,\fe])^{1-\frac{1}{n_k}}} ).
\end{equation}
To deal with the error term in (\ref{1me}), we need to use the Euler totient function for ideals $\varphi(\fn)$ defined in Section \ref{nots}. Similar to the standard Euler totient function, the function $\varphi(\fn)$ satisfies the identity $\rN\fn = \sum_{\frr | \fn} \varphi(\frr)$. 

Now we compute
\begin{eqnarray*}
 \sum_{\substack{\fd,\fe \\ \rN\fd \le w \\ \rN\fe \le y}} \frac{\log(\frac{w}{\rN\fd}) \log(\frac{y}{\rN\fe})}{\rN([\fd,\fe])^{1-\frac{1}{n_k}}} &=& \sum_{\substack{\fd,\fe \\ \rN\fd \le w \\ \rN\fe \le y}} \frac{\log(\frac{w}{\rN\fd}) \log(\frac{y}{\rN\fe})}{(\rN\fd\rN\fe)^{1-\frac{1}{n_k}}} \rN(\fd,\fe)^{1-\frac{1}{n_k}} \\
&=& \sum_{\substack{\fd,\fe \\ \rN\fd \le w \\ \rN\fe \le y}} \frac{\log(\frac{w}{\rN\fd}) \log(\frac{y}{\rN\fe})}{(\rN\fd\rN\fe)^{1-\frac{1}{n_k}}} (\sum_{\frr|(\fd,\fe)} \varphi(\frr))^{1-\frac{1}{n_k}} \\
& \le & \sum_{\substack{\fd,\fe \\ \rN\fd \le w \\ \rN\fe \le y}} \frac{\log(\frac{w}{\rN\fd}) \log(\frac{y}{\rN\fe})}{(\rN\fd\rN\fe)^{1-\frac{1}{n_k}}} (\sum_{\frr|(\fd,\fe)} \varphi(\frr)^{1-\frac{1}{n_k}}) \\ 
&=& \sum_{\substack{\frr \\ \rN\frr \le w}} \frac{\varphi(\frr)^{1-\frac{1}{n_k}}}{\rN\frr^{2-\frac{2}{n_k}}} \left( \sum_{\substack{\fm \\ \rN\fm \le \frac{w}{\rN\frr}}} \frac{\log(\frac{w}{\rN\fm\rN\frr})}{\rN\fm^{1-\frac{1}{n_k}}} \right) \left( \sum_{\substack{\fn \\ \rN\fn \le \frac{y}{\rN\frr}}} \frac{\log(\frac{y}{\rN\fn\rN\frr})}{\rN\fn^{1-\frac{1}{n_k}}} \right). 
\end{eqnarray*}
The inequality above is an application of H\"{o}lder's inequality.
By Lemma \ref{logpower} and Lemma \ref{power} (in the case $a=1$), the expression above is equal to
\begin{eqnarray}\label{firsterror}
\nonumber &=& \sum_{\substack{\frr \\ \rN\frr \le w}} \frac{\varphi(\frr)^{1-\frac{1}{n_k}}}{\rN\frr^{2-\frac{2}{n_k}}}
\left( s_kn_k^2 (\frac{w}{\rN\frr})^{\frac{1}{n_k}}+O((\log \frac{w}{\rN\frr})^2) \right) \left( s_kn_k^2 (\frac{y}{\rN\frr})^{\frac{1}{n_k}}+O((\log \frac{y}{\rN\frr})^2) \right) \\
\nonumber &=& s_k^2n_k^4 (wy)^{\frac{1}{n_k}} \sum_{\substack{\frr \\ \rN\frr \le w}} \frac{\varphi(\frr)^{1-\frac{1}{n_k}}}{\rN\frr^{2}}+O(y^{\frac{1}{n_k}}(\log w)^2 \sum_{\substack{\frr \\ \rN\frr \le w}} \frac{1}{\rN\frr}) \\
\nonumber &=& s_k^2n_k^4 (wy)^{\frac{1}{n_k}} \sum_{\substack{\frr \\ \rN\frr \le w}} \frac{\varphi(\frr)^{1-\frac{1}{n_k}}}{\rN\frr^{2}}+O(y^{\frac{1}{n_k}}(\log w)^3) \\
&=& O((wy)^{\frac{1}{n_k}}).
\end{eqnarray}

After multiplying by $x^{1-\frac{1}{n_k}}$, this shows that the error term in (\ref{1me}) is $O(x)$.

For the main term in (\ref{1me}), we compute
\begin{eqnarray*}
\sum_{\fd,\fe} \frac{\Lambda_1(\fd) \Lambda_2(\fe)}{\rN([\fd,\fe])} &=& \sum_{\fd,\fe} \frac{\Lambda_1(\fd) \Lambda_2(\fe)}{\rN(\fd\fe)} \rN((\fd,\fe)).
\end{eqnarray*}
Thus,
\begin{eqnarray*}\label{0107}
 & &  \sum_{\fd,\fe} \frac{\Lambda_1(\fd) \Lambda_2(\fe)}{\rN([\fd,\fe])} \\
\nonumber &=& \sum_{\fd,\fe} \frac{\Lambda_1(\fd) \Lambda_2(\fe)}{\rN(\fd\fe)} \sum_{\frr | (\fd,\fe)} \varphi(\frr) \\
\nonumber &=& \sum_{\rN\frr \le w} \frac{\mu_k^2(\frr) \varphi(\frr)}{\rN\frr^2} \left( \sum_{\substack{\rN\fm \le \frac{w}{\rN\frr} \\ (\fm,\frr) = (1)}} \frac{\mu_k(\fm)}{\rN\fm} \log(\frac{w}{\rN\fm\rN\frr}) \right) \left( \sum_{\substack{\rN\fn \le \frac{y}{\rN\frr} \\ (\fn,\frr) = (1)}} \frac{\mu_k(\fn)}{\rN\fn} \log(\frac{y}{\rN\fn\rN\frr}) \right).
\end{eqnarray*}

We expand the expression. To start, by Lemma \ref{10ing},
\begin{eqnarray}\label{seconderror}
\nonumber & & \sum_{\fd,\fe} \frac{\Lambda_1(\fd) \Lambda_2(\fe)}{\rN([\fd,\fe])}  \\
\nonumber &=& \sum_{\rN\frr \le w} \frac{\mu_k^2(\frr) \varphi(\frr)}{\rN\frr^2} \left( \frac{\rN\frr}{s_k\varphi(\frr)} + C_{\frr,\beta_0} (\frac{w}{\rN\frr})^{\beta_0-1}+ O(C_{\beta_0}^{'}\sigma_{-1/2}(\frr)(\log \frac{2w}{\rN\frr})^{-2}) \right)^2 \\
\nonumber &=& \frac{1}{s_k^2} \sum_{\rN\frr \le w} \frac{\mu_k^2(\frr)}{\varphi(\frr)} + w^{2\beta_0-2}\sum_{\rN\frr \le w} C_{\frr,\beta_0}^2 \frac{\mu_k^2(\frr)\varphi(\frr)}{\rN\frr^{2\beta_0}} \\
\nonumber & & + O\left( C_{\beta_0}^{'2}\sum_{\rN\frr \le w} \frac{\mu_k^2(\frr) \varphi(\frr)}{\rN\frr^2} \sigma_{-1/2}^2(\frr)(\log \frac{2w}{\rN\frr})^{-4} \right) \\
\nonumber & & +\frac{w^{\beta_0-1}}{s_k} \sum_{\rN\frr \le w} \frac{\mu_k^2(\frr)}{\rN\frr^{\beta_0}}C_{\frr,\beta_0}+ O(C_{\beta_0}^{'}\sum_{\rN\frr \le w} \frac{\mu_k^2(\frr)\sigma_{-1/2}(\frr)}{\rN\frr} (\log \frac{2w}{\rN\frr})^{-2}) \\
 & & + w^{\beta_0-1}O(C_{\beta_0}^{'}\sum_{\rN\frr \le w} \frac{\mu_k^2(\frr)\varphi(\frr)}{\rN\frr^{1+\beta_0}}|C_{\frr,\beta_0}| \sigma_{-1/2}(\frr) (\log \frac{2w}{\rN\frr})^{-2}. 
\end{eqnarray}
The term involving ${\beta_0}$ exists if and only if the Siegel zero exists (and similarly for the proofs below).

Among the six terms in (\ref{seconderror}), we apply Lemma \ref{muandphi} to the first term to obtain
\begin{equation}\label{secondmain}
\frac{1}{s_k^2} \sum_{\rN\frr \le w} \frac{\mu_k^2(\frr)}{\varphi(\frr)} = \frac{1}{s_k} \log w+O(1).
\end{equation} 
For the second term, we have
\begin{eqnarray}\label{09042}
\nonumber \left|w^{2\beta_0-2}\sum_{\rN\frr \le w} C_{\frr,\beta_0}^2 \frac{\mu_k^2(\frr)\varphi(\frr)}{\rN\frr^{2\beta_0}}\right| & \le &  w^{2\beta_0-2} \sum_{\rN\frr \le w} \frac{C_{\frr,\beta_0}^2}{\rN\frr^{2\beta_0-1}} \le w^{2\beta_0-2}C_{\beta_0}^{'2} \sum_{\rN\frr \le w} \frac{\sigma_{-\frac12}^2(\frr)}{\rN\frr^{2\beta_0-1}} \\
& \ll & \frac{C_{\beta_0}^{'2}}{1-\beta_0},
\end{eqnarray}
where $C_{\beta_0}^{'}$ is defined before Lemma \ref{10ing} and the last step uses partial summation and Lemma \ref{01073}.

We apply Lemma \ref{loginv} to the third, fifth, and sixth term in (\ref{seconderror}). For the fourth term, we apply partial summation and Lemma \ref{01073}. Then we can see that the third and fifth terms are $O(C_{\beta_0}^{'2})$, and the fourth and sixth terms are $O(\frac{C_{\beta_0}^{'2}}{1-\beta_0})$. 
Therefore,
\begin{eqnarray}\label{09048}
\nonumber \sum_{\fd,\fe} \frac{\Lambda_1(\fd) \Lambda_2(\fe)}{\rN([\fd,\fe])} &=& \frac{1}{s_k}\log w+O(\frac{C_{\beta_0}^{'2}}{1-\beta_0}) \\
&=& \frac{1}{s_k}\log w+O(\max\{\frac{1}{1-\beta_0},\frac{1}{(1-\beta_0)^5|\zeta_k^{'}(\beta_0)|^2}\}).
\end{eqnarray}

 
Combining (\ref{firsterror}) and (\ref{09048}) in (\ref{1me}), we obtain Theorem \ref{ig} in the case $wy \le x$.

\section{Proof of Theorem \ref{ig} in the case $wy > x$}\label{wygx}

Recall the definition of the von Mangoldt function for ideals in Section \ref{nots}. We have the following identity. For an ideal $\fn$, 
\begin{equation}\label{smalltrick}
 \sum_{\fd | \fn} \mu_k(\fd) \log(\frac{z}{\rN\fd}) = \begin{cases}
\Lambda(\fn) & \text{ if } \fn \neq (1), \\ 
\log z & \text{ if } \fn = (1).
\end{cases} 
\end{equation}
Using the equation above, we have for $\fn \neq (1)$,
\begin{equation*}
\sum_{\fd | \fn} \Lambda_1(\fd) = \Lambda(\fn) - \sum_{\substack{\fd | \fn \\ \rN\fd > w}} \mu_k(\fd) \log(\frac{w}{\rN\fd}) = \Lambda(\fn) + \sum_{\substack{\fd | \fn \\ \rN\fd < \frac{\rN\fn}{w}}} \mu_k(\frac{\fn}{\fd}) \log(\frac{\rN\fn}{\rN\fd w}),
\end{equation*}
\begin{equation*}
\sum_{\fe | \fn} \Lambda_2(\fe) = \Lambda(\fn) - \sum_{\substack{\fe | \fn \\ \rN\fe > y}} \mu_k(\fe) \log(\frac{y}{\rN\fe}) = \Lambda(\fn) + \sum_{\substack{\fe | \fn \\ \rN\fe < \frac{\rN\fn}{y}}} \mu_k(\frac{\fn}{\fe}) \log(\frac{\rN\fn}{\rN\fe y}). 
\end{equation*}

Now we compute the left-hand side of Theorem \ref{ig}.  We separate the terms $\rN\fn = 1$ and $1<\rN\fn \le x$ and obtain
\begin{eqnarray*}
& & \sum_{\rN\fn \le x}  \left( \sum_{\fd | \fn} \Lambda_1(\fd) \right) \left( \sum_{\fe | \fn} \Lambda_2(\fe) \right) \\ &=& \log w\log y \\
& & +\sum_{1<\rN\fn \le x}   \left( \Lambda(\fn) + \sum_{\substack{\fd | \fn \\ \rN\fd < \frac{\rN\fn}{w}}} \mu_k(\frac{\fn}{\fd}) \log(\frac{\rN\fn}{\rN\fd w}) \right) \left( \Lambda(\fn) + \sum_{\substack{\fe | \fn \\ \rN\fe < \frac{\rN\fn}{y}}} \mu_k(\frac{\fn}{\fe}) \log(\frac{\rN\fn}{\rN\fe y}) \right) \\
&=& \log w\log y + \sum_{1 < \rN\fn \le x}  \Lambda(\fn)^2 \\
& & + \sum_{1 < \rN\fn \le x}  \Lambda(\fn) \left[ \sum_{\substack{\fd | \fn \\ \rN\fd < \frac{\rN\fn}{w}}} \mu_k(\frac{\fn}{\fd}) \log(\frac{\rN\fn}{\rN\fd w}) + \sum_{\substack{\fe | \fn \\ \rN\fe < \frac{\rN\fn}{y}}} \mu_k(\frac{\fn}{\fe}) \log(\frac{\rN\fn}{\rN\fe y}) \right] \\
& & + \sum_{1 < \rN\fn \le x}  \left( \sum_{\substack{\fd | \fn \\ \rN\fd < \frac{\rN\fn}{w}}} \mu_k(\frac{\fn}{\fd}) \log(\frac{\rN\fn}{\rN\fd w}) \right) \left( \sum_{\substack{\fe | \fn \\ \rN\fe < \frac{\rN\fn}{y}}} \mu_k(\frac{\fn}{\fe}) \log(\frac{\rN\fn}{\rN\fe y}) \right) \\
& =: & \log w\log y+S_1 + S_2 + S_3.
\end{eqnarray*}

First, we discuss $S_1$.
From the proof of the effective Chebotarev density theorem, or see e.g., \cite[Lemma 12]{Mit68}, for any $A \ge 2$,
\[ \sum_{\rN\fp \le Q} \log \rN\fp = Q + O_A(Q(\log 2Q)^{-A}). \]
Then by partial summation,
\begin{eqnarray*}
 S_1=\sum_{1 < \rN\fn \le x}  \Lambda(\fn)^2 &=& \sum_{\rN\fp \le x} (\log \rN\fp)^2 + O(x) \\
 &=& \log x\sum_{\rN\fp \le x} \log \rN\fp - \sum_{\rN\fp \le x} \log \rN\fp (\log x - \log \rN\fp) \\
 &=& x\log x + O(x) + \int_1^x \frac{1}{t} \sum_{\rN\fp \le t} \log \rN\fp dt \\
 &=& x\log x + O(x).
 \end{eqnarray*}
 
 Next, we consider $S_2$. We write $\fn = \fp^m$. Then
 \[ \sum_{\substack{\fd | \fn \\ \rN\fd < \frac{\rN\fn}{w}}} \mu_k(\frac{\fn}{\fd}) \log(\frac{\rN\fn}{\rN\fd w}) = \begin{cases} -\log (\frac{\rN\fp}{w}) & \text{ if } w<\rN\fp, \\ 0 & \text{ otherwise.} \end{cases} \]
 Thus,
\[  S_2 = -\sum_{w < \rN\fp \le x} (\log \rN\fp) (\log \frac{\rN\fp}{w}) -\sum_{y < \rN\fp \le x} (\log \rN\fp) (\log \frac{\rN\fp}{y}) + O(x). \]
Using partial summation, we have
\[ \sum_{w < \rN\fp \le x} (\log \rN\fp) (\log \frac{\rN\fp}{w}) = x\log(\frac{x}{w}) + O(x) \]
and  
\[ \sum_{y < \rN\fp \le x} (\log \rN\fp) (\log \frac{\rN\fp}{y}) = x\log(\frac{x}{y}) + O(x). \]
Therefore, 
\[ S_2 = x\log(\frac{wy}{x^2}) + O(x). \]

For $S_3$, we let $(\fd,\fe) = \fq, \fd = \fq \frr, \fe = \fq \ft$ (so $(\frr,\ft)=(1)$), and let $\fn = \fm \fq \frr \ft$. Then
\begin{eqnarray}\label{s3}
 S_3 &=& \sum_{\substack{\fd \\ \rN\fd \le \frac{x}{w}}}  \sum_{\substack{\fe \\ \rN\fe \le \frac{x}{y}}} \sum_{\substack{\fn \\ \rN\fd w < \rN\fn \le x \\ \rN\fe y < \rN\fn \le x \\ [\fd,\fe] | \fn}}  \mu_k(\frac{\fn}{\fd}) \mu_k(\frac{\fn}{\fe}) \log(\frac{\rN\fn}{\rN\fd w}) \log(\frac{\rN\fn}{\rN\fe y}) \\
\nonumber &=& \sum_{\substack{\fq \\ \rN\fq \le \frac{x}{y}}} \sum_{\substack{\ft \\ \rN\ft \le \frac{x}{\rN\fq y} }} \mu_k(\ft) \sum_{\substack{\frr \\ \rN\frr \le \frac{x}{\rN\fq w} \\ (\frr,\ft)=(1)}} \mu_k(\frr) \sum_{\substack{\fm \\ \rN\fm \le \frac{x}{\rN\fq\rN\frr\rN\ft} \\ \rN\fm > \frac{w}{\rN\ft}, \rN\fm > \frac{y}{\rN\frr} \\ (\fm,\frr\ft)=(1)}}  \mu_k^2(\fm)  \log(\frac{\rN\fm\rN\ft}{w}) \log(\frac{\rN\fm\rN\frr}{y}).
\end{eqnarray}
For the innermost sum, we rewrite the summation as
\[  \sum_{\substack{\fm \\ \rN\fm \le \frac{x}{\rN\fq\rN\frr\rN\ft} \\ \rN\fm > \frac{w}{\rN\ft}, \rN\fm > \frac{y}{\rN\frr} \\ (\fm,\frr\ft)=(1)}} = \sum_{\substack{\fm \\ \rN\fm \le \frac{x}{\rN\fq\rN\frr\rN\ft} \\ (\fm,\frr\ft)=(1)}} - \sum_{\substack{\fm \\ \rN\fm \le \max\{ \frac{w}{\rN\ft}, \frac{y}{\rN\frr}\} \\ (\fm,\frr\ft)=(1)}} \]
so we need to compute, for $M > 1$, the sum
\[ \sum_{\substack{\fm \\ \rN\fm \le M \\ (\fm,\frr\ft)=(1)}} \mu_k^2(\fm)  \log(\frac{\rN\fm\rN\ft}{w}) \log(\frac{\rN\fm\rN\frr}{y}). \]

We apply partial summation twice and apply Lemma \ref{0108}. We see that the sum is equal to
\begin{eqnarray*}
& &\frac{s_k}{\zeta_k(2)} \frac{\rN\frr \rN\ft}{\kappa(\frr) \kappa(\ft)} M \left( (\log(\frac{M\rN\ft}{w})-1)(\log(\frac{M\rN\frr}{y})-1)+1  \right) \\
& + &  O\left(\log(\frac{2M\rN\ft}{w})\log(\frac{2M\rN\frr}{y})M^{1-\frac{1}{n_k}+\ep_k} \sigma_{-\frac{1}{n_k+1}}(\frr \ft)\right).
\end{eqnarray*}
Thus, the innermost sum in (\ref{s3}) is equal to
\begin{eqnarray*}
& & \frac{s_k}{\zeta_k(2)} \frac{x}{\rN\fq \kappa(\frr) \kappa(\ft)}  \left( (\log(\frac{x}{\rN\fq \rN\frr w})-1)(\log(\frac{x}{\rN\fq \rN\ft y})-1)+1  \right) \\
& + & \frac{s_k}{\zeta_k(2)} \frac{\max \{ \frac{w}{\rN\ft}, \frac{y}{\rN\frr} \} \rN\frr \rN\ft}{\kappa(\frr) \kappa(\ft)} (\log(\frac{\max \{ \frac{w}{\rN\ft}, \frac{y}{\rN\frr} \} }{\min \{ \frac{w}{\rN\ft}, \frac{y}{\rN\frr} \}}) - 2) \\
& + &  O\left(\log(\frac{2x}{\rN\fq \rN\frr w})\log(\frac{2x}{\rN\fq \rN\ft y})(\frac{x}{\rN\fq \rN\frr \rN\ft})^{1-\frac{1}{n_k}+\ep_k} \sigma_{-\frac{1}{n_k+1}}(\frr \ft)\right).
\end{eqnarray*}

Plugging this into (\ref{s3}), we have
\begin{eqnarray*}
S_3 &=& \frac{s_k}{\zeta_k(2)} x \sum_{\substack{\fq \\ \rN\fq \le \frac{x}{y}}} \frac{1}{\rN\fq} \sum_{\substack{\ft \\ \rN\ft \le \frac{x}{\rN\fq y} }} \frac{\mu_k(\ft)}{\kappa(\ft)} \sum_{\substack{\frr \\ \rN\frr \le \frac{x}{\rN\fq w} \\ (\frr,\ft)=(1)}} \frac{\mu_k(\frr)}{\kappa(\frr)} \left( (\log(\frac{x}{\rN\fq \rN\frr w})-1)(\log(\frac{x}{\rN\fq \rN\ft y})-1)+1  \right) \\
 & & + \frac{s_k}{\zeta_k(2)} \sum_{\substack{\fq \\ \rN\fq \le \frac{x}{y}}} \sum_{\substack{\frr \\ \rN\frr \le \frac{x}{\rN\fq w} }} \frac{\mu_k(\frr) \rN\frr}{\kappa(\frr)} \sum_{\substack{\ft \\ \rN\ft \le \frac{w\rN\frr}{y} \\ (\ft,\frr)=(1)}} \frac{w}{\rN\ft} \cdot \frac{\mu_k(\ft) \rN\ft}{\kappa(\ft)} \left(\log(\frac{w\rN\frr}{y\rN\ft})-2 \right) \\
 & & + \frac{s_k}{\zeta_k(2)} \sum_{\substack{\fq \\ \rN\fq \le \frac{x}{y}}} \sum_{\substack{\ft \\ \rN\ft \le \frac{x}{\rN\fq y} }} \frac{\mu_k(\ft) \rN\ft}{\kappa(\ft)} \sum_{\substack{\frr \\ \rN\frr \le \frac{y\rN\ft}{w} \\ (\frr,\ft)=(1)}} \frac{y}{\rN\frr} \cdot \frac{\mu_k(\frr) \rN\frr}{\kappa(\frr)} \left(\log(\frac{y\rN\ft}{w\rN\frr})-2 \right) \\
 & & + O\left( \sum_{\substack{\fq \\ \rN\fq \le \frac{x}{y}}} \sum_{\substack{\ft \\ \rN\ft \le \frac{x}{\rN\fq y} }} \sum_{\substack{\frr \\ \rN\frr \le \frac{x}{\rN\fq w} \\ (\frr,\ft)=(1)}}  \log(\frac{2x}{\rN\fq \rN\frr w})\log(\frac{2x}{\rN\fq \rN\ft y})(\frac{x}{\rN\fq \rN\frr \rN\ft})^{1-\frac{1}{n_k}+\ep_k} \sigma_{-\frac{1}{n_k+1}}(\frr \ft)\right) \\
 & =: & \Sigma_A+\Sigma_B+\Sigma_C+O(\Sigma_D).
\end{eqnarray*}

First, we consider $\Sigma_A$ as follows. By Lemmas \ref{11ing} and \ref{12ing}, we have
\begin{eqnarray*}
 & & \sum_{\substack{\frr \\ \rN\frr \le \frac{x}{\rN\fq w} \\ (\frr,\ft)=(1)}} \frac{\mu_k(\frr)}{\kappa(\frr)} \left( (\log(\frac{x}{\rN\fq \rN\frr w})-1)(\log(\frac{x}{\rN\fq \rN\ft y})-1)+1  \right) \\
 &=& (\log \frac{x}{\rN\fq \rN\ft y} - 1) \left[ \frac{\kappa(\ft)}{\rN\ft} \frac{\zeta_k(2)}{s_k}+C_{\ft,\beta_0}^{'}(\frac{x}{\rN\fq w})^{\beta_0-1}+O(C_{\beta_0}^{'}\sigma_{-\frac12}(\ft)(\log \frac{2x}{\rN\fq w})^{-4}) \right] \\
 & &-(\log \frac{x}{\rN\fq \rN\ft y} - 1) \left[O(C_{\ft,\beta_0}^{'}(1-\beta_0)(\frac{x}{\rN\fq w})^{\beta_0-1}+C_{\beta_0}^{'}\sigma_{-\frac12}(\ft)(\log \frac{2x}{\rN\fq w})^{-4}) \right] \\
 & & +O(C_{\ft,\beta_0}^{'}(1-\beta_0)(\frac{x}{\rN\fq w})^{\beta_0-1}+C_{\beta_0}^{'}\sigma_{-\frac12}(\ft)(\log \frac{2x}{\rN\fq w})^{-4}) \\
 &=& (\log \frac{x}{\rN\fq \rN\ft y} - 1) \left[ \frac{\kappa(\ft)}{\rN\ft} \frac{\zeta_k(2)}{s_k}+O(C_{\beta_0}^{'}\sigma_{-\frac12}(\ft)(\frac{x}{\rN\fq w})^{\beta_0-1})+O(C_{\beta_0}^{'}\sigma_{-\frac12}(\ft)(\log \frac{2x}{\rN\fq w})^{-4}) \right].
\end{eqnarray*}
Plugging this into $\Sigma_A$, we get
\begin{eqnarray}\label{sigmaa}
\nonumber \Sigma_A &=& \frac{s_k}{\zeta_k(2)}x \sum_{\substack{\fq \\ \rN\fq \le \frac{x}{y}}} \frac{1}{\rN\fq} \sum_{\substack{\ft \\ \rN\ft \le \frac{x}{\rN\fq y} }} \frac{\mu_k(\ft)}{\kappa(\ft)} \left( \frac{\kappa(\ft)}{\rN\ft} \frac{\zeta_k(2)}{s_k}\log \frac{x}{\rN\fq \rN\ft y} - \frac{\kappa(\ft)}{\rN\ft} \frac{\zeta_k(2)}{s_k} \right) \\
\nonumber & & +\frac{s_k}{\zeta_k(2)}x \sum_{\substack{\fq \\ \rN\fq \le \frac{x}{y}}} \frac{1}{\rN\fq} (\frac{x}{\rN\fq w})^{\beta_0-1} O\left( C_{\beta_0}^{'}  \sum_{\substack{\ft \\ \rN\ft \le \frac{x}{\rN\fq y} }} \sigma_{-\frac12}(\ft) \frac{|\log \frac{x}{\rN\fq \rN\ft y}-1|}{\rN\ft} \right) \\
\nonumber & & +\frac{s_k}{\zeta_k(2)}x \sum_{\substack{\fq \\ \rN\fq \le \frac{x}{y}}} \frac{1}{\rN\fq} (\log \frac{2x}{\rN\fq w})^{-4} O\left(C_{\beta_0}^{'}\sum_{\substack{\ft \\ \rN\ft \le \frac{x}{\rN\fq y} }} \frac{|\log \frac{x}{\rN\fq \rN\ft y}-1|}{\rN\ft} \sigma_{-\frac12}(\ft)  \right) \\
 &=& x\sum_{\substack{\fq \\ \rN\fq \le \frac{x}{y}}} \frac{1}{\rN\fq} \sum_{\substack{\ft \\ \rN\ft \le \frac{x}{\rN\fq y} }} \frac{\mu_k(\ft)}{\rN\ft} (\log \frac{x}{\rN\fq \rN\ft y}-1) \\
\nonumber & & + O\left(C_{\beta_0}^{'} x\sum_{\substack{\fq \\ \rN\fq \le \frac{x}{y}}} \frac{1}{\rN\fq} (\log \frac{x}{\rN\fq y}) (\frac{x}{\rN\fq w})^{\beta_0-1} \sum_{\substack{\ft \\ \rN\ft \le \frac{x}{\rN\fq y} }} \frac{1}{\rN\ft} \sigma_{-\frac12}(\ft)\right)  \\
\nonumber & & + O\left(C_{\beta_0}^{'}x\sum_{\substack{\fq \\ \rN\fq \le \frac{x}{y}}} \frac{1}{\rN\fq} (\log \frac{2x}{\rN\fq w})^{-3} \sum_{\substack{\ft \\ \rN\ft \le \frac{x}{\rN\fq y} }} \frac{1}{\rN\ft} \sigma_{-\frac12}(\ft)  \right).
\end{eqnarray}


Applying (\ref{7ing}) and (\ref{8ing}) to the inner sum of the first term in (\ref{sigmaa}), we have
\begin{eqnarray*}
 & & \sum_{\substack{\ft \\ \rN\ft \le \frac{x}{\rN\fq y} }} \frac{\mu_k(\ft)}{\rN\ft} (\log \frac{x}{\rN\fq \rN\ft y} - 1) \\
 &=& (\log \frac{x}{\rN\fq y}-1) \left[ \frac{(\frac{x}{\rN\fq y})^{\beta_0-1}}{(\beta_0-1)\zeta_k^{'}(\beta_0)}+O((\log \frac{2x}{\rN\fq y})^{-3}) \right] \\
 & & - \left[ -\frac{1}{s_k}+\frac{(\frac{x}{\rN\fq y})^{\beta_0-1}((\beta_0-1)\log \frac{x}{\rN\fq y}-1)}{(\beta_0-1)^2\zeta_k^{'}(\beta_0)}+O((\log \frac{2x}{\rN\fq y})^{-2})\right] \\
&=& \frac{1}{s_k}+\frac{(\frac{x}{\rN\fq y})^{\beta_0-1}}{(\beta_0-1)^2\zeta_k^{'}(\beta_0)}(2-\beta_0)+O((\log \frac{2x}{\rN\fq y})^{-2}).
\end{eqnarray*}
Thus, by Lemma \ref{power} (in the cases $a=1$ and $a=\beta_0$) and Lemma \ref{loginv}, the first term in (\ref{sigmaa}) satisfies
\[ x\sum_{\substack{\fq \\ \rN\fq \le \frac{x}{y}}} \frac{1}{\rN\fq} \sum_{\substack{\ft \\ \rN\ft \le \frac{x}{\rN\fq y} }} \frac{\mu_k(\ft)}{\rN\ft} (\log(\frac{x}{\rN\ft \rN\fq y}) - 1) = x\log \frac{x}{y}+O(\frac{1}{(1-\beta_0)^3|\zeta_k^{'}(\beta_0)|}x). \]

Next, we consider the second term in (\ref{sigmaa}). Note that for any positive integer $A$ and positive number $a$, we have 
\begin{equation}\label{beta0small}(\beta_0-1)a+A\log a \le  A\log \frac{A}{1-\beta_0}.
\end{equation}
Then for any $\fq$, we apply (\ref{beta0small}) with $A=4, a=\log \frac{x}{\rN\fq w}$ and obtain $(\frac{x}{\rN\fq w})^{\beta_0-1} \ll \frac{1}{(1-\beta_0)^4} (\log \frac{x}{\rN\fq w})^{-4} \ll \frac{1}{(1-\beta_0)^4} (\log \frac{x}{\rN\fq y})^{-4}$. Then we can consider the second term and the third term in (\ref{sigmaa}) together as
\[ O\left(\frac{C_{\beta_0}^{'}}{(1-\beta_0)^4}x\sum_{\substack{\fq \\ \rN\fq \le \frac{x}{y}}} \frac{1}{\rN\fq} (\log \frac{2x}{\rN\fq w})^{-3} \sum_{\substack{\ft \\ \rN\ft \le \frac{x}{\rN\fq y} }} \frac{1}{\rN\ft} \sigma_{-\frac12}(\ft)  \right). \]
Applying partial summation and Lemma \ref{01073}  on the innermost sum of this term, we have
\[ \sum_{\substack{\ft \\ \rN\ft \le \frac{x}{\rN\fq y} }} \frac{1}{\rN\ft} \sigma_{-\frac12}(\ft) \ll \log \frac{2x}{\rN\fq y} \le \log \frac{2x}{\rN\fq w}. \]
Thus, by Lemma \ref{loginv}, the second and the third terms in (\ref{sigmaa}) are $O(\frac{C_{\beta_0}^{'}}{(1-\beta_0)^4}x)$.

Combining all the estimates for the terms in (\ref{sigmaa}), we have 
\[ \Sigma_A=x\log \frac{x}{y}+O\left((\frac{1}{(1-\beta_0)^3|\zeta_k^{'}(\beta_0)|}+\frac{C_{\beta_0}^{'}}{(1-\beta_0)^4})x\right). \]

Second, we consider $\Sigma_B$ and $\Sigma_C$. Our approach is essentially the same as that in estimating $\Sigma_A$. 
By Lemmas \ref{11ing} and \ref{12ing}, we have
\begin{eqnarray*}
& & \sum_{\substack{\ft \\ \rN\ft \le \frac{w\rN\frr}{y} \\ (\ft,\frr)=(1)}} \frac{\mu_k(\ft)}{\kappa(\ft)} (\log \frac{w\rN\frr}{y\rN\ft}-2) \\
 &=& \left( \frac{\kappa(\frr)}{\rN\frr} \frac{\zeta_k(2)}{s_k} + C_{\frr,\beta_0}^{'}(\frac{w\rN\frr}{y})^{\beta_0-1} + O(C_{\beta_0}^{'}\sigma_{-\frac12}(\frr)(\log \frac{2w\rN\frr}{y})^{-2}) \right) \\
& & +O\left( |C_{\frr,\beta_0}^{'}|(1-\beta_0)(\frac{w\rN\frr}{y})^{\beta_0-1}+C_{\beta_0}^{'}\sigma_{-\frac12}(\frr)(\log \frac{2w\rN\frr}{y})^{-2} \right).
\end{eqnarray*}
Plugging this into $\Sigma_B$, we get
\begin{eqnarray}\label{sigmab}
\nonumber \Sigma_B &=& w\sum_{\substack{\fq \\ \rN\fq \le \frac{x}{y}}} \sum_{\substack{\frr \\ \frac{y}{w} \le \rN\frr \le \frac{x}{\rN\fq w} }} \mu_k(\frr) + O\left( C_{\beta_0}^{'} w \sum_{\substack{\fq \\ \rN\fq \le \frac{x}{y}}} \sum_{\substack{\frr \\ \frac{y}{w} \le \rN\frr \le \frac{x}{\rN\fq w} }} (\frac{w\rN\frr}{y})^{\beta_0-1}\sigma_{-\frac12}(\frr) \right)  \\
& & +O\left(C_{\beta_0}^{'}w \sum_{\substack{\fq \\ \rN\fq \le \frac{x}{y}}} \sum_{\substack{\frr \\ \frac{y}{w} \le \rN\frr \le \frac{x}{\rN\fq w} }}  \sigma_{-\frac12}(\frr) (\log \frac{2w\rN\frr}{y})^{-2} \right).
\end{eqnarray}
By (\ref{6ing}), Lemma \ref{power}, and Lemma \ref{loginv}, the first term in (\ref{sigmab}) is bounded by
\begin{eqnarray*}
&  \ll & w\sum_{\substack{\fq \\ \rN\fq \le \frac{x}{y}}} \frac{(\frac{x}{\rN\fq w})^{\beta_0}}{\beta_0|\zeta_k^{'}(\beta_0)|} + w \sum_{\substack{\fq \\ \rN\fq \le \frac{x}{y}}} \frac{x}{\rN\fq w} (\log \frac{2x}{\rN\fq w})^{-2} \\
& \ll &  \frac{w(\frac{x}{w})^{\beta_0}(\frac{x}{y})^{1-\beta_0}}{(1-\beta_0)|\zeta_k^{'}(\beta_0)|} + x \sum_{\substack{\fq \\ \rN\fq \le \frac{x}{y}}} \frac{1}{\rN\fq} (\log \frac{2x}{\rN\fq y})^{-2} \ll \frac{x}{(1-\beta_0)|\zeta_k^{'}(\beta_0)|}. 
\end{eqnarray*}
By (\ref{beta0small}) with $A=2, a=\log \frac{2w\rN\frr}{y}$, we have $(\frac{w\rN\frr}{y})^{\beta_0-1} \ll \frac{1}{(1-\beta_0)^2} (\log \frac{2w\rN\frr}{y})^{-2}$. Thus, the two error terms in (\ref{sigmab}) can be combined as
\[ O\left(\frac{C_{\beta_0}^{'}}{(1-\beta_0)^2}w \sum_{\substack{\fq \\ \rN\fq \le \frac{x}{y}}} \sum_{\substack{\frr \\ \frac{y}{w} \le \rN\frr \le \frac{x}{\rN\fq w} }}  \sigma_{-\frac12}(\frr) (\log \frac{2w\rN\frr}{y})^{-2} \right). \]
By partial summation, Lemma \ref{01073}, and Lemma \ref{loginv}, this term is bounded by
\[ \ll \frac{C_{\beta_0}^{'}}{(1-\beta_0)^2}w \sum_{\substack{\fq \\ \rN\fq \le \frac{x}{y}}} \frac{x}{\rN\fq w} (\log \frac{2x}{\rN\fq y})^{-2} \ll \frac{C_{\beta_0}^{'}}{(1-\beta_0)^2}x \sum_{\substack{\fq \\ \rN\fq \le \frac{x}{y}}} \frac{1}{\rN\fq} (\log \frac{2x}{\rN\fq y})^{-2} \ll \frac{C_{\beta_0}^{'}}{(1-\beta_0)^2}x. \]

 Thus, 
 \[ \Sigma_B = O\left((\frac{1}{(1-\beta_0)|\zeta_k^{'}(\beta_0)|}+\frac{C_{\beta_0}^{'}}{(1-\beta_0)^2})x\right). \] Similarly, $\Sigma_C = O\left((\frac{1}{(1-\beta_0)|\zeta_k^{'}(\beta_0)|}+\frac{C_{\beta_0}^{'}}{(1-\beta_0)^2})x\right)$. 
 
 Third, we consider $\Sigma_D$. We write
 \begin{eqnarray*}
 \Sigma_D &=& x^{1-\frac{1}{n_k}+\ep_k} \sum_{\substack{\fq \\ \rN\fq \le \frac{x}{y}}} \rN\fq^{-1+\frac{1}{n_k}-\ep_k} \sum_{\substack{\ft \\ \rN\ft \le \frac{x}{\rN\fq y} }} \rN\ft^{-1+\frac{1}{n_k}-\ep_k} \log(\frac{2x}{\rN\fq \rN\ft y}) \sigma_{-\frac{1}{n_k+1}}(\ft) \\
 & & \times \sum_{\substack{\frr \\ \rN\frr \le \frac{x}{\rN\fq w} \\ (\frr,\ft)=(1)}} \rN\frr^{-1+\frac{1}{n_k}-\ep_k} \log(\frac{2x}{\rN\fq \rN\frr w})  \sigma_{-\frac{1}{n_k+1}}(\frr) \\
 & \ll & x^{1-\frac{1}{n_k}+\ep_k} \sum_{\substack{\fq \\ \rN\fq \le \frac{x}{y}}} \rN\fq^{-1+\frac{1}{n_k}-\ep_k} \sum_{\substack{\ft \\ \rN\ft \le \frac{x}{\rN\fq y} }} \rN\ft^{-1+\frac{1}{n_k}-\ep_k} \log(\frac{2x}{\rN\fq \rN\ft y}) \sigma_{-\frac{1}{n_k+1}}(\ft) \\
 & & \times \sum_{\substack{\frr \\ \rN\frr \le \frac{x}{\rN\fq w}}} \rN\frr^{-1+\frac{1}{n_k}-\ep_k} \log(\frac{2x}{\rN\fq \rN\frr w})  \sigma_{-\frac{1}{n_k+1}}(\frr). 
 \end{eqnarray*}
For the sum involving $\ft$ and $\frr$, they have the same form, so we compute the following sum, for $Q>1$,
\begin{eqnarray*}
 \sum_{\substack{\fd \\ \rN\fd \le Q}} \rN\fd^{-1+\frac{1}{n_k}-\ep_k} \log(\frac{2Q}{\rN\fd})  \sigma_{-\frac{1}{n_k+1}}(\fd) 
 & \ll & \sum_{k \ge 1} k (\frac{Q}{2^k})^{-1+\frac{1}{n_k}-\ep_k} \sum_{\substack{\fd \\ \frac{Q}{2^k} < \rN\fd \le \frac{2Q}{2^k} }} \sigma_{-\frac{1}{n_k+1}}(\fd) \\
 & \ll & \sum_{k \ge 1} k (\frac{Q}{2^k})^{\frac{1}{n_k}-\ep_k} \\
 & \ll & Q^{\frac{1}{n_k}-\ep_k}.
 \end{eqnarray*}
 Therefore,
 \begin{eqnarray*}
 \Sigma_D & \ll & x^{1-\frac{1}{n_k}+\ep_k} \sum_{\substack{\fq \\ \rN\fq \le \frac{x}{y}}} \rN\fq^{-1+\frac{1}{n_k}-\ep_k} \cdot (\frac{x}{\rN\fq y})^{\frac{1}{n_k}-\ep_k} \cdot (\frac{x}{\rN\fq w})^{\frac{1}{n_k}-\ep_k} \\
 & = & x^{1+\frac{1}{n_k}-\ep_k} (wy)^{-\frac{1}{n_k}+\ep_k} \sum_{\rN\fq \le \frac{x}{y}} \rN\fq^{-1-\frac{1}{n_k}+\ep_k} \ll x^{1+\frac{1}{n_k}-\ep_k} (wy)^{-\frac{1}{n_k}+\ep_k} \\
 & \ll & x.
 \end{eqnarray*}
 Combining our computations for $\Sigma_A, \Sigma_B,\Sigma_C,\Sigma_D$, we have
 \begin{eqnarray*}
  S_3 &=& x\log \frac{x}{y} + O\left((\frac{1}{(1-\beta_0)^3|\zeta_k^{'}(\beta_0)|}+\frac{C_{\beta_0}^{'}}{(1-\beta_0)^4})x\right) \\
  &=& x\log \frac{x}{y}+O\left(\max\{ \frac{1}{(1-\beta_0)^4},\frac{1}{(1-\beta_0)^6|\zeta_k^{'}(\beta_0)|} \} x\right). 
  \end{eqnarray*}
 Combining our computations for $S_1,S_2,S_3$, we obtain Theorem \ref{ig} in the case $wy>x$.

 \section{Proofs of lemmas and deductions}\label{pflem}
 
 In this section, we do three things so that all of our proofs are complete. First, we prove Theorem \ref{igtoug1} assuming Theorem \ref{ig}. Second, we prove Corollary \ref{ug} assuming Theorem \ref{igtoug1}. Third, we prove all the lemmas stated in Section \ref{lems}.

 \begin{proof}[Proof of Theorem \ref{igtoug1} assuming Theorem \ref{ig}]
We compute
\begin{eqnarray*}
(\log(\frac{y}{w}))^2 \sum_{\rN\fn \le x} \left( \sum_{\fd | \fn} \lambda_\fd \right)^2 &=& \sum_{\rN\fn \le x} \left( \sum_{\substack{\fd | \fn \\ \rN\fd \le y}} \mu_k(\fd)\log(\frac{y}{\rN\fd}) - \sum_{\substack{\fd | \fn \\ \rN\fd \le w}} \mu_k(\fd)\log(\frac{w}{\rN\fd}) \right)^2 \\
&=& \sum_{\rN\fn \le x} \left( \sum_{\substack{\fd | \fn \\ \rN\fd \le w}} \Lambda_1(\fd) \right)^2 + \sum_{\rN\fn \le x} \left( \sum_{\substack{\fe | \fn \\ \rN\fe \le y}} \Lambda_2(\fe) \right)^2 \\
& &- 2\sum_{\rN\fn \le x} \left( \sum_{\substack{\fd | \fn \\ \rN\fd \le w}} \Lambda_1(\fd) \right) \left( \sum_{\substack{\fe | \fn \\ \rN\fe \le y}} \Lambda_2(\fe) \right).
\end{eqnarray*}
Applying Theorem \ref{ig} to each of the three terms, we obtain Theorem \ref{igtoug1}.
\end{proof}

\begin{proof}[Proof of Corollary \ref{ug} assuming Theorem \ref{igtoug1}]

Define 
\[ \Delta(\fn) := \sum_{\fd | \fn} \lambda_\fd. \] 
We compute the left-hand side of the statement in Corollary \ref{ug}. Note that it does not affect the result whether we start the sum from 1 or from $w$ because $\Delta(\fn)=0$ if $1<\fn<w$. Precisely:

\begin{eqnarray*}
\sum_{\rN\fn \le x}  \Delta(\fn)^2 \rN\fn^{1-2\alpha} 
&=& x^{1-2\alpha} \sum_{\rN\fn \le x}  \Delta(\fn)^2 - \sum_{\rN\fn \le x}  \Delta(\fn)^2 (x^{1-2\alpha} - \rN\fn^{1-2\alpha}) \\
&=& x^{1-2\alpha} \sum_{\rN\fn \le x}  \Delta(\fn)^2 - \sum_{\rN\fn \le x}  \Delta(\fn)^2 \int_{\rN\fn}^x (1-2\alpha) t^{-2\alpha} dt +1 \\
&=& x^{1-2\alpha} \sum_{\rN\fn \le x}  \Delta(\fn)^2 + (2\alpha-1) \int_1^x t^{-2\alpha} \left( \sum_{n=1}^t \Delta(\fn)^2 \right) dt+1.
\end{eqnarray*}
Theorem \ref{igtoug1} tells us that
\[ \sum_{\rN\fn \le x} \Delta(\fn)^2 = \frac{x}{\log(y/w)}+O(\frac{C_{\beta_0}x}{\log^2(y/w)}). \]
The constant associated to the error term may depend on $n_k,d_k$, and we omit the notations in this proof.

Thus,
\begin{eqnarray*}
\sum_{\rN\fn \le x}  \Delta(\fn)^2 \rN\fn^{1-2\alpha} &=& \frac{x^{2-2\alpha}}{\log(y/w)} + (2\alpha-1) \int_1^x t^{-2\alpha} \left( \frac{t}{\log(y/w)} \right) dt + O(\frac{C_{\beta_0}x^{2-2\alpha}}{\log^2(y/w)}) \\
&=& \frac{x^{2-2\alpha}}{2-2\alpha} \cdot \frac{1}{\log(y/w)}+O(\frac{C_{\beta_0}x^{2-2\alpha}}{\log^2(y/w)}).
\end{eqnarray*}
Similarly,
\[ \sum_{\rN\fn <w}  \Delta(\fn)^2 \rN\fn^{1-2\alpha} = \frac{w^{2-2\alpha}}{2-2\alpha} \cdot \frac{1}{\log(y/w)}+O(\frac{C_{\beta_0}w^{2-2\alpha}}{\log^2(y/w)}). \]
Therefore,
\begin{eqnarray*}
 \sum_{w \le \rN\fn \le x} \Delta(\fn)^2 \rN\fn^{1-2\alpha} &=& \frac{1}{\log(y/w)} \cdot \frac{x^{2-2\alpha} - w^{2-2\alpha}}{2-2\alpha} (1+O(\frac{C_{\beta_0}}{\log(y/w)}) \\
 & \ll & \frac{C_{\beta_0}}{\log(y/w)} \cdot \frac{x^{2-2\alpha} - w^{2-2\alpha}}{2-2\alpha}
 \end{eqnarray*} 
Since
\[ \frac{x^{2-2\alpha} - w^{2-2\alpha}}{2-2\alpha} = \int_{\log w}^{\log x} e^{(2-2\alpha)t} dt \le \log(x/w)x^{2-2\alpha}, \]
this completes the proof of Corollary \ref{ug} assuming Theorem \ref{igtoug1}.
\end{proof}

\begin{proof}[Proof of Lemma \ref{power}]
Using partial summation and applying Theorem \ref{counting}, we write
\begin{eqnarray*}
\sum_{\substack{\rN\fn \le Q}} \rN\fn^{-a} &=& Q^{-a} \sum_{\substack{\rN\fn \le Q}} 1-\sum_{\substack{\rN\fn \le Q}} (Q^{-a}-\rN\fn^{-a}) \\
&=& s_k Q^{1-a}-\sum_{\substack{\rN\fn \le Q}} \int_{\rN\fn}^Q (-a)t^{-1-a} dt + O(Q^{1-\frac{1}{n_k}-a}) \\
&=& s_k Q^{1-a}+\int_1^Q at^{-1-a}\sum_{\substack{\rN\fn \le t}} 1 dt + O(Q^{1-\frac{1}{n_k}-a}) \\
&=& s_k Q^{1-a}+\int_1^Q as_kt^{-a} dt +O(\int_1^Q at^{-\frac{1}{n_k}-a}dt) + O(Q^{1-\frac{1}{n_k}-a}).
\end{eqnarray*}

If $0<a< 1-\frac{1}{n_k}$ or $1-\frac{1}{n_k} < a <1$, it is $\frac{s_k}{1-a}Q^{1-a}+ O(Q^{1-\frac{1}{n_k}-a})+O(1).$

If $a = 1-\frac{1}{n_k}$, it is $s_kn_kQ^{\frac{1}{n_k}} + O(\log Q).$

If $a = 1$, it is $s_k \log Q + O(1).$
\end{proof}

\begin{proof}[Proof of Lemma \ref{logpower}]
Lemma \ref{logpower} follows from Lemma \ref{power} (in the case $a = 1-\frac{1}{n_k}$) and partial summation.
\end{proof}

\begin{proof}[Proof of Lemma \ref{01073}]
The first inequality holds because $\sigma_a(\fn) \ge 1$ by definition.

 Recall the function $d(\fn)$ in Section \ref{nots}. For any $\ep>0$, $d(\fn) \ll \rN\fn^\ep$.
We have
\begin{eqnarray*}
 \sum_{\rN\fn \le Q} \sigma_a^2(\fn) &=& \sum_{\rN\fn \le Q} \sum_{\fd,\fe | \fn} \rN\fd^a \rN\fe^a = \sum_{\substack{\rN\fd \le Q \\ \rN\fe \le Q}} \rN\fd^a \rN\fe^a \sum_{\substack{\rN\fn \le Q \\ [\fd,\fe] | \fn}} 1 \ll Q\sum_{\fd,\fe} \rN\fd^a \rN\fe^a \rN[\fd,\fe]^{-1} \\
 & \le & Q\sum_{\fd,\fe} \rN[\fd,\fe]^{-1+a} \le Q\sum_{\fq} d^2(\fq)\rN\fq^{-1+a} \ll_a Q. 
 \end{eqnarray*}
\end{proof}

\begin{proof}[Proof of Lemma \ref{loginv}]
We only prove the first statement. The second statement follows from the same method.
Using partial summation, we write
\begin{eqnarray*}
\sum_{\rN\fn \le Q} \frac{1}{\rN\fn (\log(\frac{2Q}{\rN\fn}))^2} &=& \sum_{\rN\fn \le Q} \frac{1}{Q(\log (\frac{2Q}{Q}))^2} - \sum_{\rN\fn \le Q} \left[ \frac{1}{Q (\log(\frac{2Q}{Q}))^2} - \frac{1}{\rN\fn (\log(\frac{2Q}{\rN\fn}))^2} \right] \\
&=& O(1) + \sum_{\rN\fn \le Q} \int_{\rN\fn}^Q \frac{\log 2Q-\log t-2}{t^2(\log 2Q-\log t)^3} dt \\
&=& \int_1^Q \frac{\log 2Q-\log t-2}{t^2(\log 2Q-\log t)^3} \left( \sum_{\rN\fn \le t} 1 \right) dt + O(1) \\
&=& O\left( \int_1^Q \frac{\log 2Q-\log t-2}{t(\log 2Q-\log t)^3} dt \right) \\
&=& O\left(\int_1^Q \frac{1}{t(\log 2Q-\log t)^2} dt \right) + O\left(\int_1^Q \frac{2}{t(\log 2Q-\log t)^3} dt \right) \\
&=& O(1).
\end{eqnarray*}
\end{proof}

\begin{proof}[Proof of Lemma \ref{muandphi}]
Let $a_\fn = \mu_k^2(\fn)\prod_{\fp | \fn} (1-\frac{1}{\rN\fp})^{-1}$. Then for $\Re(s)>1$, we have
\begin{equation}\label{azetab}
\sum_{\fn} \frac{a_\fn}{\rN\fn^s} = \prod_\fp (1+\frac{1}{\rN\fp^s}(1-\frac{1}{\rN\fp})^{-1}) = \zeta_k(s) \sum_{\fm} \frac{b_\fm}{\rN\fm^s}
\end{equation}
where
\begin{equation*}
\sum_{\fm} \frac{b_\fm}{\rN\fm^s} = \prod_{\fp} (1+\frac{1}{\rN\fp^s}(1-\frac{1}{\rN\fp})^{-1})(1-\frac{1}{\rN\fp^s}) = \prod_\fp (1+\frac{1}{\rN\fp^s} \cdot \frac{1}{\rN\fp - 1} - \frac{1}{\rN\fp^{2s}} \cdot \frac{\rN\fp}{\rN\fp - 1})
\end{equation*}
and $\displaystyle \sum_{\fm} \frac{b_\fm}{\rN\fm^s}$ is absolutely convergent for $\Re(s) > \frac12$. Here, the sequence $\{a_\fn\}$ can be regarded as the convolution of the constant sequence $1$ and the sequence $\{b_\fm\}$. Let $\displaystyle F(Q)=\frac{1}{s_k} \sum_{\substack{\fd \\ \rN\fd \le Q}} 1$ and let $\displaystyle G_v(Q)=\sum_{\rN\fm \le Q} |b_\fm|$. Then
\[ F(Q) = Q+O(Q^{1-\frac{1}{n_k}}), \ \ \ G_v(Q) = O(Q^{\frac12 + \ep}) \text{ for any } \ep>0. \] 
The sum $\sum_{\rN\fn \le Q} a_\fn$ can be estimated from the bounds of $F(Q)$ and $G_v(Q)$ via the stability theorem, e.g., \cite[Theorem 3.29]{BD04}. In particular, we have
\begin{equation*}
\sum_{\rN\fn \le Q} a_\fn = AQ+\begin{cases} O(Q^{\frac12+\ep}) \text{ for any } \ep>0 & n_k=2, \\ O(Q^{1-\frac{1}{n_k}}) & n_k  \ge 3, \end{cases}
\end{equation*}
where $\displaystyle A=\lim_{s \to 1^{+}} (s-1) \sum_{\fn} \frac{a_\fn}{\rN\fn^s}.$ Since $\displaystyle \sum_{\fm} \frac{b_\fm}{\rN\fm^s} \to 1$ as $s \to 1^{+}$, we have $A=\mathrm{Res}_{s=1} \zeta_k(s)=s_k$ by (\ref{azetab}).

Lemma \ref{muandphi} follows from partial summation:
\[ \sum_{\rN\fn \le Q} \frac{\mu_k^2(\fn)}{\varphi(\fn)} = \sum_{\rN\fn \le Q} \frac{a_\fn}{\rN\fn} = Q^{-1} \sum_{\rN\fn \le Q} a_\fn + \int_1^Q u^{-2} \left( \sum_{\rN\fn \le u} a_\fn \right) du = s_k \log Q+O(1). \]
\end{proof}

\begin{proof}[Proof of Lemma \ref{sqfrc}]
We prove this lemma by the inclusion-exclusion principle. 

Let $\cA$ be the set of all squarefree ideals in $\cO_k$ with norm at most $Q$. For a set $X$ in $\{\fn: \rN\fn \le Q\}$, denote $\overline{X}$ to be the complement of $X$. For any prime ideal $\fq$, let $\cA_\fq = \{ \fn: \fq | \fn, \rN\fn \le Q \}$ and let $M$ denote the largest number of prime ideals $\fq_1,\dots,\fq_M$ such that $\rN\fq_1 \dots \rN\fq_M \le \sqrt{Q}$. Then by the inclusion-exclusion principle, we have
\begin{equation}\label{cala}
 |\cA| = \left| \bigcap_{\fq \text{ prime }} \overline{\cA_\fq} \right| = \sum_{m=0}^M (-1)^m \sum_{\fq_1,\dots,\fq_m} \left| \cA_{\fq_1} \cap \dots \cap \cA_{\fq_m}  \right|. 
 \end{equation}
Denote $S:=\cA_{\fq_1} \cap \dots \cap \cA_{\fq_m}$. Then 
\[ S=\{ \rN\fn \le Q: \fq_1^2 \dots \fq_m^2 | \fn \}=\{ \frr:  \rN\frr \le \frac{Q}{\rN\fq_1 \dots \rN\fq_m} \}=\frac{s_kQ}{\rN\fq_1^2 \dots \rN\fq_m^2}+O((\frac{Q}{\rN\fq_1^2 \dots \rN\fq_m^2})^{1-\frac{1}{n_k}}). \] Plugging this into (\ref{cala}), we have
\begin{eqnarray*}
|\cA| &=& \sum_{m=0}^M (-1)^m \sum_{\fq_1,\dots,\fq_m} \frac{s_kQ}{\rN\fq_1^2 \dots \rN\fq_m^2} + O(\sum_{m=0}^M \sum_{\fq_1,\dots,\fq_m} (\frac{Q}{\rN\fq_1^2 \dots \rN\fq_m^2})^{1-\frac{1}{n_k}} ) \\
&=& s_kQ \sum_{\rN\fd \le \sqrt{Q}} \frac{\mu_k(\fd)}{\rN\fd^2} + Q^{1-\frac{1}{n_k}}O(\sum_{\rN\fd \le \sqrt{Q}} \frac{1}{\rN\fd^{2-\frac{2}{n_k}}}) \\
&=& s_kQ \sum_{\fd} \frac{\mu_k(\fd)}{\rN\fd^2} - s_kQ \sum_{\rN\fd > \sqrt{Q}} \frac{\mu_k(\fd)}{\rN\fd^2} + Q^{1-\frac{1}{n_k}}O(\sum_{\rN\fd \le \sqrt{Q}} \frac{1}{\rN\fd^{2-\frac{2}{n_k}}}) \\
&=& \frac{s_kQ}{\zeta_k(2)} + Q^{1-\frac{1}{n_k}}O(\sum_{\rN\fd \le \sqrt{Q}} \frac{1}{\rN\fd^{2-\frac{2}{n_k}}}).
\end{eqnarray*}
If $n_k=2$, $|\cA| = \frac{s_kQ}{\zeta_k(2)}+O(Q^{\frac12+\ep_k})$ for any $\ep_k>0$; if $n_k \ge 3$, $|\cA| = \frac{s_kQ}{\zeta_k(2)} + O(Q^{1-\frac{1}{n_k}}) $.
\end{proof}

\begin{proof}[Proof of Lemma \ref{0108}]
For any ideal $\frr$, let 
\[ \cD_\frr := \{ \fd: \fp | \fd \Rightarrow \fp | \frr \}. \]
Then we have
\[ \sum_{\substack{\fd,\fm \\ \fd \in \cD_\frr \\ \fd\fm = \fn}} \lambda(\fd)\mu_k^2(\fm) = \begin{cases} \mu_k^2(\fn) & \text{ if } (\fn,\frr)=(1), \\ 0 & \text{ if } (\fn,\frr) \neq (1). \end{cases} \]
Thus,
\begin{eqnarray*}
\sum_{\substack{\rN\fn \le Q \\ (\fn,\frr) = (1)}} \mu_k^2(\fn)  &=& \sum_{\fd \in \cD_\frr} \lambda(\fd) \sum_{\rN\fm \le \frac{Q}{\rN\fd}} \mu_k^2(\fm) \\
&=& \frac{s_kQ}{\zeta_k(2)} \sum_{\fd \in \cD_\frr} \frac{\lambda(\fd)}{\rN\fd} + O(Q^{1-\frac{1}{n_k}+\ep_k} \sum_{\fd \in \cD_\frr} \rN\fd^{-1+\frac{1}{n_k}-\ep_k}) \\
&=& \frac{s_k}{\zeta_k(2)} \frac{\rN\frr}{\kappa(\frr)}Q + O(Q^{1-\frac{1}{n_k}+\ep_k} \prod_{\fp | \frr} \frac{1}{1-\rN\fp^{-1+\frac{1}{n_k}-\ep_k}}) \\
&=& \frac{s_k}{\zeta_k(2)} \frac{\rN\frr}{\kappa(\frr)}Q + O(Q^{1-\frac{1}{n_k}+\ep_k} \prod_{\fp | \frr} (1+\frac{1}{\rN\fp^{1-\frac{1}{n_k}+\ep_k}-1})) \\
&=& \frac{s_k}{\zeta_k(2)} \frac{\rN\frr}{\kappa(\frr)}Q +  O(Q^{1-\frac{1}{n_k}+\ep_k} \sigma_{-\frac{1}{n_k+1}}(\frr)).
\end{eqnarray*}
\end{proof}

\begin{proof}[Proof of Lemma \ref{678ing}]
For (\ref{6ing}), we will prove the following stronger form:
\begin{equation}\label{stronger}
\sum_{\rN\fn \le Q} \mu_k(\fn)-\frac{Q^{\beta_0}}{\beta_0\zeta_k^{'}(\beta_0)} \ll Q\exp(-\frac{c_1}{2}(\log Q)^{1/10}n_k^{-2}),
\end{equation}
where $c_1=\min\{c,10^{-6}\}$, $c$ is the absolute constant in Theorem \ref{ik04}.
The term $\frac{Q^{\beta_0}}{\beta_0\zeta_k^{'}(\beta_0)}$ exists if and only if the Siegel zero exists.
Note that this form is equivalent to the effective prime ideal theorem (see e.g., \cite{TZ17}). For completeness, we give a proof of (\ref{stronger}). 
The approach of the proof is similar to the proof of Theorem 1.2 in \cite{FM12}. However, the error term in \cite[Theorem 1.2]{FM12} appears to be ineffective, and \cite[p.7]{FM12} appears to omit the bound for the integral from $-t_0$ to $t_0$. Thus, we rewrite the proof as follows.

A direct approach to prove (\ref{stronger}) has a problem since there may be many ideals with the same norm. Thus, we apply the following result of Liu and Ye.

\begin{theorem}\label{ly}\cite[Theorem 2.1]{LY07}
Let $f(s)=\sum_{n=1}^\infty a_n n^{-s}$ converge absolutely for $\sigma>\sigma_0>0$, and let $B(\sigma)=\sum_{n=1}^\infty |a_n|n^{-\sigma}$. Then for $b>\sigma_0$, $x \ge 2$, $T \ge 2$, and $H \ge 2$, we have
\[ \sum_{n \le x} a_n = \frac{1}{2\pi i} \int_{b-iT}^{b+iT} f(s) \frac{x^s}{s}ds+O\left(\sum_{x-x/H<n \le x+x/H} |a_n| \right)+O\left(\frac{x^bHB(b)}{T} \right). \]
\end{theorem}

Let $a=1-\frac{c_1}{n_k^2 (\log |d_k| (T+3)^{n_k})^9}$, 
where $c_1$ is the absolute constant we defined in the beginning of the proof.
We apply Theorem \ref{ly} for $x=Q$, $f(s)=\frac{1}{\zeta_k(s)}, b=1+\frac{1}{\log Q}, H=\sqrt{T}$ to get
\begin{eqnarray*} 
\sum_{\rN\fq \le Q} \mu_k(\fq) &=& \frac{1}{2\pi i} \int_{1+\frac{1}{\log Q}-iT}^{1+\frac{1}{\log Q}+iT} \frac{1}{\zeta_k(s)} \frac{Q^s}{s} ds \\
& & +O((Q+\frac{Q}{\sqrt{T}})-(Q-\frac{Q}{\sqrt{T}})+Q^{1-\frac{1}{n_k}})+O\left( \frac{Q\zeta_k(1+\frac{1}{\log Q})}{\sqrt{T}} \right). 
\end{eqnarray*}
For the last term, we apply $\zeta_k(s)=\frac{s_k}{s-1}+O(1)$ to get $\zeta_k(1+\frac{1}{\log Q})=O(\log Q)$. We will choose $T$ to satisfy $T \ll Q^\ep$ for any $\ep>0$, so the last error term dominates and we have
\begin{equation}\label{stoi30}
 \sum_{\rN\fq \le Q} \mu_k(\fq)=\frac{1}{2\pi i} \int_{1+\frac{1}{\log Q}-iT}^{1+\frac{1}{\log Q}+iT} \frac{1}{\zeta_k(s)} \frac{Q^s}{s} ds+O\left( \frac{Q\log Q}{\sqrt{T}} \right). 
 \end{equation}

Let $\psi(t)=n_k^2(\log |d_k| (|t|+3)^{n_k})^9$, so that $a=1-\frac{c_1}{\psi(T)}$. Put
 \[ l(t)=\begin{cases} 1-\frac{c_1}{\psi(10)} &  |t| \le 10, \\ 1-\frac{c_1}{\psi(t)} & 10 \le |t| \le T. \end{cases} \]
 By the residue theorem and Theorem \ref{ik04}, we have
\begin{eqnarray*}
\nonumber \frac{1}{2\pi i} \int_{1+\frac{1}{\log Q}-iT}^{1+\frac{1}{\log Q}+iT} \frac{1}{\zeta_k(s)} \frac{Q^s}{s} ds &=& \mathrm{Res}_{s=\beta_0} \frac{1}{\zeta_k(s)} \frac{Q^{\beta_0}}{\beta_0} \\
& & + \frac{1}{2\pi i} \left(\int_{b-iT}^{a-iT}+\int_{l(t),\ t=-T}^{t=T}+\int_{a+iT}^{b+iT} \right) \frac{1}{\zeta_k(s)}  \frac{Q^s}{s} ds \\
 & =: &  \frac{Q^{\beta_0}}{\beta_0\zeta_k^{'}(\beta_0)}+ I_1+I_2+I_3.
\end{eqnarray*}
The term $\frac{Q^{\beta_0}}{\beta_0\zeta_k^{'}(\beta_0)}$ exists if and only if the Siegel zero exists.


To proceed, a bound for $\frac{1}{\zeta_k(s)}$ is necessary. We will obtain the bound for $\frac{1}{\zeta_k(s)}$ from a bound for $\frac{\zeta_k^{'}}{\zeta_k}(s)$, which we can find in \cite[p.656]{Lan03}.
  Let $\sigma_1=1-\frac{c_1}{\psi(10)}$. 
 If the Siegel zero $\beta_0$ exists and $\sigma_1 < \beta_0 < 1$, we cannot let $\sigma_1$ and $\beta_0$ be arbitrarily close (see (\ref{sigmabetaaway}) below). We can achieve this by the following method. There are two possibilities: either $\frac{1+\sigma_1}{2} < \beta_0$ or $\frac{1+\sigma_1}{2} \ge \beta_0$. If we are in the latter case, we choose $c_1^{'}=\frac{c_1}{4}$ and $\sigma_1^{'}=1-\frac{c_1^{'}}{\psi(10)}$. Let $c_1^{'},\sigma_1^{'}$ be our new values of $c_1,\sigma_1$. Then our new value satisfies $\beta_0<\sigma_1$. In this case, our argument below still holds and all the terms involving $\beta_0$ are omitted, since we will only consider the region to the right of $\sigma_1$. Therefore, we can assume that our choice of $\sigma_1$ satisfies $\frac{1+\sigma_1}{2} < \beta_0$ (i.e., $\beta_0-\sigma_1>\frac12(1-\sigma_1)$) without loss of generality.

 Let $s=\sigma+it$.
 Then $-\frac{\zeta_k^{'}}{\zeta_k}(s)$ is regular for $\sigma \ge l(t)$ except at $s=1$ 
 and $s=\beta_0$ (if the Siegel zero exists and satisfies $\beta_0-\sigma_1>\frac12(1-\sigma_1)$) with both residues equal to $1$ and $-1$, respectively, and
 \[ -\frac{\zeta_k^{'}}{\zeta_k}(s) = \begin{cases} \frac{-1}{s-\beta_0}+\frac{1}{s-1}+O(1) & l(t) \le \sigma \le 2, |t| \le 10, \\ O(\psi(t)) & l(t) \le \sigma \le 2, |t| \ge 10. \end{cases} \]
 Thus, there exists an absolute constant $C$ such that $\left| \frac{\zeta_k^{'}}{\zeta_k}(s) \right| \le C\psi(t)$ for $s$ in the range $l(t) \le \sigma \le 2, |t| \ge 10$.
 
 To obtain a bound of $\frac{1}{\zeta_k(s)}$, we compute for $\Re(s)=\sigma \ge 1-\frac{c_1}{\psi(\Im(s))}$, $|\Im(s)|=t \ge 10$, following \cite[p. 311]{Ivi12},
 \begin{eqnarray*}
  \log \frac{1}{|\zeta_k(s)|} &=& -\Re \log \zeta_k(s) \\
  &=& -\Re \log \zeta_k(1+\frac{1}{\psi(t)}+it)+\int_\sigma^{1+\frac{1}{\psi(t)}} \Re \frac{\zeta_k^{'}}{\zeta_k}(u+it)du \\
  & \le & \log \zeta_k(1+\frac{1}{\psi(t)})+\int_\sigma^{1+\frac{1}{\psi(t)}} C\psi(t) du \\
  & \le & \log C\psi(t) + O(1).
  \end{eqnarray*}
  For $\Re(s)=\sigma =\sigma_1= 1-\frac{c_1}{\psi(10)}$, $0<|\Im(s)|=t \le 10$, we compute (recalling $\beta_0-\sigma>\frac12(1-\sigma)$)
   \begin{eqnarray}\label{sigmabetaaway}
\nonumber  \log \frac{1}{|\zeta_k(s)|} &=& -\Re \log \zeta_k(s) \\
\nonumber  &=& -\Re \log \zeta_k(1+\frac{1}{\psi(10)}+it)+\int_\sigma^{1+\frac{1}{\psi(10)}} \Re \frac{\zeta_k^{'}}{\zeta_k}(u+it)du \\
 \nonumber & \le & \log \zeta_k(1+\frac{1}{\psi(10)})+\int_\sigma^{1+\frac{1}{\psi(10)}} \Re(\frac{-1}{u+it-\beta_0}+\frac{1}{u+it-1}) du+O(1) \\
  &=& \frac12  (\log((1+\frac{1}{\psi(10)}-\beta_0)^2+t^2)-\log((\sigma-\beta_0)^2+t^2)) \\
 \nonumber & &+ \frac12 (\log((\frac{1}{\psi(10)})^2+t^2)-\log((\sigma-1)^2+t^2))  +O(1) \\
 \nonumber & = & O(1) + O(\log(c_1^2+t^2)) \\
 \nonumber & = & O(1).
  \end{eqnarray}
  Therefore, 
  \begin{equation}\label{zetainvbd}
  \frac{1}{\zeta_k(s)} \ll \begin{cases} \psi(t) &  l(t) \le \sigma \le 2, |t| \ge 10, \\ 1 & \sigma = l(t)=1-\frac{c_1}{\psi(10)}, |t| \le 10. \end{cases}
  \end{equation}

 We apply (\ref{zetainvbd}) to estimate $I_1,I_2,I_3$. For $I_1$,
 \begin{equation}\label{i1bd}
  |I_1| \le \frac{1}{2\pi} \int_a^b \left| \frac{1}{\zeta_k(\sigma-iT)} \right| \frac{Q^\sigma}{|\sigma-iT|} d\sigma \ll (\log T)^{9}T^{-1} \int_a^b Q^\sigma d\sigma \ll QT^{-1}(\log T)^9. 
  \end{equation}
 The same bound holds for $I_3$. For $I_2$,
 \begin{eqnarray}\label{i2bd}
\nonumber  |I_2| & \le &  \frac{1}{2\pi} \left( \int_{l(t),t=-T}^{t=-10}+\int_{l(t), t=-10}^{t=10}+\int_{l(t),t=10}^{t=T} \right) \left| \frac{1}{\zeta_k(\sigma+it)} \right| \frac{Q^\sigma}{|\sigma+it|} ds \\
& \ll &  Q^a  \int_{10}^T (\log T)^{9} \frac{1}{|a+it|} dt \ll Q^a (\log T)^{10}. 
  \end{eqnarray}
  Our choice of $T$ will balance the big O terms in (\ref{stoi30}) and (\ref{i2bd}). Note that the term in (\ref{i1bd}) is $\ll \frac{Q\log Q}{\sqrt{T}}$, which is the term in (\ref{stoi30}). Upon choosing
  \begin{equation}\label{tvalue}
   T=\exp(\frac{(\log Q)^{\frac{1}{10}}-\log |d_k|}{n_k})-3, 
   \end{equation}
  we see that the term in (\ref{stoi30}) is bounded above by
  \[ Q\log Q\cdot  \exp(-\frac{(\log Q)^{\frac{1}{10}}-\log |d_k|}{n_k}) \ll Q\exp(-\frac12(\log Q)^{1/10}n_k^{-1})   \]
  and the term in (\ref{i2bd}) is bounded above by
  \begin{eqnarray*}
   Q^{1-\frac{c_1}{n_k^2 (\log (|d_k| \exp((\log Q)^{\frac{1}{10}}-\log |d_k|)))^9}}(\log T)^{10}&=&Q^{1-\frac{c_1}{n_k^2 (\log Q)^{\frac{9}{10}}}}(\log T)^{10} \\
   & \ll & Q\exp(-\frac{c_1}{2}(\log Q)^{1/10}n_k^{-2}). 
   \end{eqnarray*}
  Since we assumed $T \ge 10$, by (\ref{tvalue}), our bound holds for $Q \ge \exp((n_k \log 13+\log |d_k|)^{10})$. If $Q < \exp((n_k \log 13+\log |d_k|)^{10})$, then $\sum_{\rN\fn \le Q} \mu_k(\fn) \ll 1$.
  This finishes the proof of (\ref{stronger}), hence (\ref{6ing}) holds.


For (\ref{7ing}),
we follow the proof of (\ref{6ing}) but choose $f(s)=\frac{1}{\zeta_k(s+1)}$ 
 instead. The Dirichlet series is $f(s)=\sum_\fq a_\fq=\frac{\mu_k(\fq)}{\rN\fq^{s+1}}$. In this new setting, we choose $b=\frac{1}{\log Q}$. Then the proof of (\ref{6ing}) is adapted here.

For (\ref{8ing}),
we follow the proof of (\ref{6ing}) but choose $f(s)=\frac{\zeta_k^{'}}{\zeta_k^2}(s+1)$
 instead. The Dirichlet series is $f(s)=\sum_\fq \frac{\mu_k(\fq) \log \rN\fq}{\rN\fq^{s+1}}$. In this new setting, we choose $b=\frac{1}{\log Q}$. Then the proof of (\ref{6ing}) is adapted here. Note that there is a pole at $s=0$, and the residue is 
\[ \text{Res}_{s=0} \frac{\zeta_k^{'}}{\zeta_k^2}(s+1) \cdot \frac{Q^s}{s}=\text{Res}_{s=1} \frac{\zeta_k^{'}}{\zeta_k^2}(s)=-\frac{1}{\text{Res}_{s=1} \zeta_k(s)}=-\frac{1}{s_k}. \]
There is also a possible pole of order 2 at $s=\beta_0-1$ if the Siegel zero exists, which gives the term $\frac{Q^{\beta_0-1}[(\beta_0-1)\log Q-1]}{(\beta_0-1)^2\zeta_k^{'}(\beta_0)}$. Note also that the bound for $\frac{\zeta_k^{'}}{\zeta_k^2}(s+1)$ comes from multiplying the bound for $\frac{\zeta_k^{'}}{\zeta_k}(s+1)$ and the bound for $\frac{1}{\zeta_k}(s+1)$.


\end{proof}

\begin{proof}[Proof of Lemma \ref{10ing}]
For any integral ideal $\frr$, again let 
\[ \cD_\frr := \{ \fd: \fp | \fd \Rightarrow \fp | \frr \}. \]
Then it is clear that
\begin{equation*}
\sum_{\substack{\fd \in \cD_\frr \\ \fd | \fn}} \mu_k(\frac{\fn}{\fd}) = 
\begin{cases} \mu_k(\fn) & (\fn,\frr) = (1), \\ 0 & (\fn,\frr) \neq (1). 
\end{cases} 
\end{equation*}
Thus,
\[ \sum_{\substack{\rN\fn \le Q \\ (\fn,\frr) = (1)}} \frac{\mu_k(\fn)}{\rN\fn} \log(\frac{Q}{\rN\fn}) = \sum_{\fd \in \cD_\frr} \frac{1}{\rN\fd} \sum_{\rN\fm \le \frac{Q}{\rN\fd}} \frac{\mu_k(\fm)}{\rN\fm} \log(\frac{Q}{\rN\fm\rN\fd}). \]
By (\ref{7ing}) and (\ref{8ing}) in Lemma \ref{678ing}, the above is equal to
\begin{eqnarray}\label{01072}
\nonumber &=& \sum_{\substack{\fd \in \cD_\frr \\ \rN\fd \le Q}} \frac{1}{\rN\fd} (\frac{1}{s_k} +\frac{(\frac{Q}{\rN\fd})^{\beta_0-1}}{(\beta_0-1)^2\zeta_k^{'}(\beta_0)}+ O_A((\log \frac{2Q}{\rN\fd})^{-A})) \\
\nonumber &=& \frac{1}{s_k} \sum_{\fd \in \cD_\frr} \frac{1}{\rN\fd} +\frac{\sum_{\substack{\fd \in \cD_\frr}} \rN\fd^{-\beta_0}}{(\beta_0-1)^2\zeta_k^{'}(\beta_0)} Q^{\beta_0-1} \\
\nonumber & &+ O(\sum_{\substack{\fd \in \cD_\frr \\ \rN\fd > Q}} \frac{1}{\rN\fd})+O(\frac{\sum_{\substack{\fd \in \cD_\frr \\ \rN\fd > Q}} \rN\fd^{-\beta_0}}{(\beta_0-1)^2\zeta_k^{'}(\beta_0)} Q^{\beta_0-1})+O_A(\sum_{\substack{\fd \in \cD_\frr \\ \rN\fd \le Q}} \frac{1}{\rN\fd} (\log \frac{2Q}{\rN\fd})^{-A}). \\
& &
\end{eqnarray}
Here the first term is equal to $\frac{\rN\frr}{s_k\varphi(\frr)}$. Let 
\[ C_{\frr,\beta_0}=\frac{\sum_{\substack{\fd \in \cD_\frr}} \rN\fd^{-\beta_0}}{(\beta_0-1)^2\zeta_k^{'}(\beta_0)}. \]

Recall the definition of $\sigma_{a}(\fn)$ in Section \ref{nots}. Note that for $(\fn,\fm) = (1)$, we have 
\[ \sigma_{a}(\fn\fm) = \sigma_a(\fn)\sigma_a(\fm). \] Then the first error term in (\ref{01072}) is
\begin{eqnarray}\label{sumdr}
\nonumber & \ll & Q^{-1/3} \sum_{\fd \in \cD_\frr} \rN\fd^{-2/3} = Q^{-1/3} \prod_{\fp | \frr} (1+(\rN\fp^{2/3}-1)^{-1}) \ll Q^{-1/3} \prod_{\fp | \frr} (1+\rN\fp^{-1/2}) \\
 & \ll_A & (\log 2Q)^{-A}\sigma_{-1/2}(\frr)
\end{eqnarray}
for all but finitely many $\fp$. Similarly, for the second error term in (\ref{01072}), it is bounded by
\[ \ll_A \frac{1}{(\beta_0-1)^2|\zeta_k^{'}(\beta_0)|} (\log 2Q)^{-A}\sigma_{-1/2}(\frr). \]
For the third error term in (\ref{01072}), we separate the sum as
\begin{eqnarray*}
 \sum_{\substack{\fd \in \cD_\frr \\ \rN\fd \le Q}} \frac{1}{\rN\fd} (\log \frac{2Q}{\rN\fd})^{-A} &=& \sum_{\substack{\fd \in \cD_\frr \\ \rN\fd \le Q^{1/2}}} \frac{1}{\rN\fd} (\log \frac{2Q}{\rN\fd})^{-A} + \sum_{\substack{\fd \in \cD_\frr \\ Q^{1/2} < \rN\fd \le Q}} \frac{1}{\rN\fd} (\log \frac{2Q}{\rN\fd})^{-A} \\
 & \ll & (\log 2Q)^{-A}\sum_{\fd \in \cD_\frr} \rN\fd^{-1} + Q^{-1/6} \sum_{\fd \in \cD_\frr} \rN\fd^{-2/3} \\
 & \ll & (\log 2Q)^{-A}\sigma_{-1/2}(\frr).
 \end{eqnarray*}
 Note also that using the approach in (\ref{sumdr}), we have
\[ |C_{\frr,\beta_0}|= \frac{\sum_{\substack{\fd \in \cD_\frr}} \rN\fd^{-\beta_0}}{(\beta_0-1)^2|\zeta_k^{'}(\beta_0)|} \ll \frac{\sigma_{-\frac12}(\frr)}{(\beta_0-1)^2|\zeta_k^{'}(\beta_0)|}. \]
 This finishes the proof of Lemma \ref{10ing}.
\end{proof}

\begin{proof}[Proof of Lemma \ref{11ing}]
We first prove the following formula. For any ideal $\fn$, 
\[ \frac{1}{\kappa(\fn)} = \frac{1}{\rN\fn} \sum_{\fd | \fn} \frac{\mu_k(\fd)}{\kappa(\fd)}. \]
It holds because
\[ \frac{1}{\rN\fn} \sum_{\fd | \fn} \frac{\mu_k(\fd)}{\kappa(\fd)} = \frac{1}{\rN\fn} \sum_{\fd | \fn} \mu_k(\fd) \prod_{\fp | \fd} \frac{1}{\rN\fp+1} = \frac{1}{\rN\fn} \prod_{\fp | \fn} (1-\frac{1}{\rN\fp + 1}) = \frac{1}{\kappa(\fn)}. \]
Then we are able to compute the left-hand side as
\begin{eqnarray}\label{02011} 
\nonumber \sum_{\substack{\rN\fn \le Q \\ (\fn,\frr) = (1)}} \frac{\mu_k(\fn)}{\kappa(\fn)} \log(\frac{Q}{\rN\fn}) &=& \sum_{\substack{\rN\fd \le Q \\ (\fd,\frr) = (1)}} \frac{\mu_k(\fd)}{\kappa(\fd)} \sum_{\substack{\rN\fn \le Q \\ (\fn,\frr) = (1) \\ \fd | \fn}}  \frac{\mu_k(\fn)}{\rN\fn} \log(\frac{Q}{\rN\fn}) \\
&=& \sum_{\substack{\rN\fd \le Q \\ (\fd,\frr) = (1)}} \frac{\mu_k^2(\fd)}{\rN\fd \kappa(\fd)} \sum_{\substack{\rN\fm \le \frac{Q}{\rN\fd} \\ (\fm,\fd\frr) = (1)}}  \frac{\mu_k(\fm)}{\rN\fm} \log(\frac{Q}{\rN\fm \rN\fd}). 
\end{eqnarray}
By Lemma \ref{10ing}, (\ref{02011}) is equal to
\begin{eqnarray*}
&=& \frac{\rN\frr}{s_k \varphi(\frr)} \sum_{\substack{\rN\fd \le Q \\ (\fd,\frr) = (1)}} \frac{\mu_k^2(\fd)}{\kappa(\fd) \varphi(\fd)} +\sum_{\substack{\rN\fd \le Q \\ (\fd,\frr) = (1)}}C_{\fd\frr,\beta_0} \frac{\mu_k^2(\fd)}{\rN\fd \kappa(\fd)} (\frac{Q}{\rN\fd})^{\beta_0-1} \\
& & + O_A(C_{\beta_0}^{'}\sigma_{-\frac12}(\frr) \sum_{d \le Q} \rN\fd^{-2} (\log \frac{2Q}{\rN\fd})^{-A}) \\
&=& \frac{\rN\frr}{s_k \varphi(\frr)} \prod_{\fp \nmid \frr} (1-\rN\fp^{-2})^{-1} +\sum_{\substack{(\fd,\frr) = (1)}}C_{\fd\frr,\beta_0} \frac{\mu_k^2(\fd)}{\rN\fd \kappa(\fd)} (\frac{Q}{\rN\fd})^{\beta_0-1} \\
& & + O(\frac{\rN\frr}{\varphi(\frr)} \sum_{\rN\fd > Q} \rN\fd^{-2}) + O_A(C_{\beta_0}^{'}\sigma_{-\frac12}(\frr) \sum_{d \le Q} \rN\fd^{-2} (\log \frac{2Q}{\rN\fd})^{-A}) \\
&=& \frac{\kappa(\frr)}{\rN\frr} \frac{\zeta_k(2)}{s_k} +C_{\frr,\beta_0}^{'}Q^{\beta_0-1}+ O_A(C_{\beta_0}^{'}\sigma_{-\frac12}(\frr) (\log 2Q)^{-A})
\end{eqnarray*}
for a constant $C_{\frr,\beta_0}^{'}$ depending only on $\frr,\beta_0$.

Note also that
\begin{eqnarray*}
|C_{\frr,\beta_0}^{'}| &=& \left| \sum_{\substack{(\fd,\frr) = (1)}}C_{\fd\frr,\beta_0} \frac{\mu_k^2(\fd)}{\rN\fd \kappa(\fd)} (\frac{1}{\rN\fd})^{\beta_0-1} \right| \le \sum_{\substack{(\fd,\frr) = (1)}}|C_{\fd\frr,\beta_0}| \frac{1}{\rN\fd^2} (\frac{1}{\rN\fd})^{\beta_0-1} \\
& \le & \sum_{\substack{(\fd,\frr) = (1)}}|C_{\beta_0}^{'}\sigma_{-\frac12}(\fd\frr)| \frac{1}{\rN\fd^2} (\frac{1}{\rN\fd})^{\beta_0-1} = C_{\beta_0}^{'}\sigma_{-\frac12}(\frr) \sum_{\substack{(\fd,\frr) = (1)}}\sigma_{-\frac12}(\fd) \frac{1}{\rN\fd^2} (\frac{1}{\rN\fd})^{\beta_0-1} \\
& \ll & C_{\beta_0}^{'}\sigma_{-\frac12}(\frr) \sum_{\substack{(\fd,\frr) = (1)}}\rN\fd^{\frac12} \frac{1}{\rN\fd^2} (\frac{1}{\rN\fd})^{\beta_0-1}  \ll  C_{\beta_0}^{'}\sigma_{-\frac12}(\frr).
\end{eqnarray*}
This finishes the proof of Lemma \ref{11ing}.
\end{proof}

\begin{proof}[Proof of Lemma \ref{12ing}]
For any constant $c>1$, 
\begin{eqnarray*}
(\log c) \sum_{\substack{\rN\fn \le Q \\ (\fn,\frr) = (1)}} \frac{\mu_k(\fn)}{\kappa(\fn)} &=& \sum_{\substack{\rN\fn \le Qc \\ (\fn,\frr) = (1)}} \frac{\mu_k(\fn)}{\kappa(\fn)} \log (\frac{Qc}{\rN\fn}) - \sum_{\substack{\rN\fn \le Q \\ (\fn,\frr) = (1)}} \frac{\mu_k(\fn)}{\kappa(\fn)} \log (\frac{Q}{\rN\fn}) \\
& & - \sum_{\substack{Q < \rN\fn \le Qc \\ (\fn,\frr) = (1)}} \frac{\mu_k(\fn)}{\kappa(\fn)} \log (\frac{Qc}{\rN\fn}).
\end{eqnarray*} 
We take $c=1+(\log 2Q)^{-A} \asymp \exp((\log 2Q)^{-A})$. By Lemma \ref{11ing}, the first two sums on the right-hand side is 
\begin{eqnarray*}
& = & C_{\frr,\beta_0}^{'}Q^{\beta_0-1}((1+(\log 2Q)^{-A})^{\beta_0-1}-1)+O_A(C_{\beta_0}^{'}\sigma_{-\frac12}(\frr) (\log 2Q)^{-2A}) \\
&=& O(C_{\frr,\beta_0}^{'}Q^{\beta_0-1}|\exp((\beta_0-1)(\log 2Q)^{-A})-1|) + O_A(C_{\beta_0}^{'}\sigma_{-\frac12}(\frr) (\log 2Q)^{-2A}) \\
&=& O(C_{\frr,\beta_0}^{'}Q^{\beta_0-1}(1-\beta_0)(\log 2Q)^{-A}) + O_A(C_{\beta_0}^{'}\sigma_{-\frac12}(\frr) (\log 2Q)^{-2A}).
\end{eqnarray*}
The third sum is $\ll (\log c) \sum_{Q<\rN\fn \le Qc} \rN\fn^{-1} \ll (\log c)^2$. Therefore,
\begin{eqnarray*}
 \sum_{\substack{\rN\fn \le Q \\ (\fn,\frr) = (1)}} \frac{\mu_k(\fn)}{\kappa(\fn)} & \ll_A  & \frac{C_{\frr,\beta_0}^{'}Q^{\beta_0-1}(1-\beta_0)(\log 2Q)^{-A}}{\log c} +C_{\beta_0}^{'}\frac{\sigma_{-\frac12}(\frr) (\log 2Q)^{-2A}}{\log c} + \log c \\
 &\asymp& C_{\frr,\beta_0}^{'}Q^{\beta_0-1}(1-\beta_0) +C_{\beta_0}^{'}\sigma_{-\frac12}(\frr) (\log 2Q)^{-A}+ (\log 2Q)^{-A} \\
 & \asymp & C_{\frr,\beta_0}^{'}Q^{\beta_0-1}(1-\beta_0) +C_{\beta_0}^{'}\sigma_{-\frac12}(\frr) (\log 2Q)^{-A}.
 \end{eqnarray*}
\end{proof}

This completes the proofs of all lemmas in Section \ref{lems} and hence Theorems \ref{ig} and \ref{igtoug1}, and Corollary \ref{ug} are proved.




\subsection*{Acknowledgements}
The contents in this paper are part of the author's Ph.D. dissertation at Duke University. The author thanks his advisor Lillian Pierce for her helpful and consistent guidance during this work, and Jesse Thorner for suggesting this topic.

\bibliographystyle{alpha}

\bibliography{NoThBibliography}

\begin{thebibliography}{LDTT22}

\bibitem[An20]{An20}
Chen An.
\newblock Log-free zero density estimates for automorphic ${L}$-functions.
\newblock {\em arXiv preprint arXiv:2004.14410}, 2020.

\bibitem[BD04]{BD04}
Paul~Trevier Bateman and Harold~G Diamond.
\newblock {\em Analytic number theory: an introductory course}, volume~1.
\newblock World Scientific, 2004.

\bibitem[BV68]{BV68}
M.~B. Barban and P.~P. Vehov.
\newblock An extremal problem.
\newblock {\em Trudy Moskov. Mat. Ob\v{s}\v{c}.}, 18:83--90, 1968.

\bibitem[FM12]{FM12}
Yusuke Fujisawa and Makoto Minamide.
\newblock On partial sums of the {M}\"{o}bius and {L}iouville functions for
  number fields.
\newblock {\em arXiv preprint arXiv:1212.4348}, 2012.

\bibitem[Gra78]{Gr78}
S.~Graham.
\newblock An asymptotic estimate related to {S}elberg's sieve.
\newblock {\em Journal of Number Theory}, 10(1):83--94, 1978.

\bibitem[Gra81]{Gr81}
S.~Graham.
\newblock On {L}innik's constant.
\newblock {\em Acta Arithmetica}, 2(39):163--179, 1981.

\bibitem[IK04]{IK04}
Henryk Iwaniec and Emmanuel Kowalski.
\newblock {\em Analytic number theory}, volume~53.
\newblock American Mathematical Soc., 2004.

\bibitem[Ivi12]{Ivi12}
Aleksandar Ivic.
\newblock {\em The Riemann zeta-function: theory and applications}.
\newblock Courier Corporation, 2012.

\bibitem[Jut79]{Ju79}
Matti Jutila.
\newblock On a problem of {B}arban and {V}ehov.
\newblock {\em Mathematika}, 26(1):62--71, 1979.

\bibitem[Lan03]{Lan03}
Edmund Landau.
\newblock Neuer {B}eweis des {P}rimzahlsatzes und {B}eweis des
  {P}rimidealsatzes.
\newblock {\em Math. Ann.}, 56(4):645--670, 1903.

\bibitem[LDTT22]{LTT22}
David Lowry-Duda, Takashi Taniguchi, and Frank Thorne.
\newblock Uniform bounds for lattice point counting and partial sums of zeta
  functions.
\newblock {\em Mathematische Zeitschrift}, 300(3):2571--2590, 2022.

\bibitem[LY07]{LY07}
Jianya Liu and Yangbo Ye.
\newblock Perron's formula and the prime number theorem for automorphic
  ${L}$-functions.
\newblock {\em Pure and Applied Mathematics Quarterly}, 3(2):481--497, 2007.

\bibitem[Mit68]{Mit68}
Takayoshi Mitsui.
\newblock On the prime ideal theorem.
\newblock {\em Journal of the Mathematical Society of Japan}, 20(1-2):233--247,
  1968.

\bibitem[Mur18]{Mur19}
V.~Kumar Murty.
\newblock The {B}arban-{V}ehov theorem in arithmetic progressions.
\newblock {\em Hardy-Ramanujan J.}, 41:157--171, 2018.

\bibitem[TZ17]{TZ17}
Jesse Thorner and Asif Zaman.
\newblock An explicit bound for the least prime ideal in the chebotarev density
  theorem.
\newblock {\em Algebra \& Number Theory}, 11(5):1135--1197, 2017.

\end{thebibliography}






\end{document}